\documentclass[11pt,a4paper]{amsart}
\usepackage{mathptmx} 
\usepackage[scaled=0.90]{helvet} 
\usepackage{courier}
\usepackage{amssymb}
\usepackage{graphicx}
\usepackage{color}
\usepackage{tikz-cd}
\normalfont
\usepackage[T1]{fontenc}

\usepackage{amssymb}

\renewcommand{\epsilon}{\varepsilon}
\renewcommand{\setminus}{\smallsetminus}
\renewcommand{\emptyset}{\varnothing}

\newtheorem{theorem}{Theorem}[section]
\newtheorem{proposition}[theorem]{Proposition}
\newtheorem{corollary}[theorem]{Corollary}
\newtheorem{lemma}[theorem]{Lemma}

\theoremstyle{definition}

\theoremstyle{remark}
\newtheorem{remark}[theorem]{Remark}

\newcommand{\lk}{\mathrm{lk}}

\def\Vightarrow#1{\smash{\mathop{\longrightarrow}\limits^{#1}}}

\newcommand{\normal}{\lhd}

\newcommand{\Char}{\mathrm{char}}

\newcommand{\Ker}{\mathrm{Ker}}
\newcommand{\gr}{\mathrm{gr}}

\newcommand{\Q}{\mathbb Q}
\newcommand{\Z}{\mathbb Z}
\newcommand{\N}{\mathbb N}
\newcommand{\R}{\mathbb R}

\newcommand{\FP}{\operatorname{FP}}

\newcommand{\Ho}{\operatorname{H}}
\newcommand{\cohom}[3]{H^{{\raise1pt\hbox{$\scriptstyle#1$}}}(#2\>\!,#3)}
\newcommand{\tatecohom}[3]%
  {\widehat H^{{\raise1pt\hbox{$\scriptstyle#1$}}}(#2\>\!,#3)}

\newcommand{\Cohom}[3]%
  {H^{{\raise1pt\hbox{$\scriptstyle#1$}}}\big(#2\>\!,#3\big)}
\newcommand{\Tatecohom}[3]%
  {\widehat H^{{\raise1pt\hbox{$\scriptstyle#1$}}}\big(#2\>\!,#3\big)}

\newcommand{\homol}[3]{{\mathrm H}_{{\lower1pt\hbox{$\scriptstyle#1$}}}(#2\>\!,#3)}
\newcommand{\homolog}[2]{{\mathrm H}_{{\lower1pt\hbox{$\scriptstyle#1$}}}(#2)}

\renewcommand{\ker}{\operatorname{Ker}}
\newcommand{\im}{\operatorname{Im}}

\newcommand{\Tor}{\operatorname{Tor}}

\newcommand{\GG}{\mathfrak g}

\title[Coabelian ideals in right angled Artin  Lie algebras]{Coabelian ideals in  $\mathbb{N}$-graded Lie algebras  and applications to  right angled Artin Lie algebras}
\author{D.~ H. ~Kochloukova}

\address{Dessislava H.~Kochloukova, Department of Mathematics, University of Campinas, 
13083 - 859 Campinas, SP, Brazil}\email{desi@unicamp.br}

\author{C.~Mart\'inez-P\'erez}

\address{Conchita Mart\'inez-P\'erez, Departamento de Matem\'aticas, Universidad de Zaragoza,
50009 Zaragoza, Spain} \email{conmar@unizar.es}


\thanks{}

\begin{document}

\thispagestyle{empty}

\begin{abstract} We consider homological  finiteness properties $FP_n$  of certain $\mathbb{N}$-graded Lie algebras. After proving some general results,  see Theorem A, Corollary B and Corollary C,   we concentrate on a family that can be considered as the Lie algebra version of the generalized Bestvina-Brady  groups  associated to a graph $\Gamma$.  We prove that the homological finiteness properties of these  Lie  algebras can be determined in terms of the graph in the same way as in the group case.
\end{abstract}

\maketitle

\section{Introduction}
In this paper we study homological properties of certain $\mathbb{N}$-graded Lie algebras. Recall that a Lie algebra $L$ is $\mathbb{N}$-graded if  $L = \oplus_{i \geq 1} L_i$ for $L_i$ such that $[L_i, L_j] \subseteq L_{i+j}$ for $i,j \geq 1$. If not otherwise stated  the Lie algebras we consider are over a field $K$ of arbitrary characteristic.

In Section \ref{homgrad} we study homological finiteness properties of $\mathbb{N}$-graded Lie algebras $L$. Recall that a Lie algebra is of homological type $FP_m$ if the trivial $U(L)$-module $K$ has a projective resolution, where all projectives in dimension smaller or equal to $m$ are finitely generated.  If such a projective resolution exists then there is a free resolution with the same property. We note that a Lie algebra $L$ is of type $FP_1$ precisely when $L$ is finitely generated and if $L$ is finitely presented (in terms of generators and relations) then $L$ is of type $FP_2$. It is an open problem whether  there exist a Lie algebra that is of type $FP_2$ but is not finitely presented.  But  by \cite[(2.13)]{Weigel} for $\mathbb{N}$-graded Lie algebras the properties $FP_2$ and finite presentability coincide,  see for more details Lemma \ref{FP2}.    We show in Proposition \ref{homgrad2} that $\mathbb{N}$-graded Lie algebras also behave  like pro-$p$ groups in the sense that  $L$ is of homological type $FP_m$ if and only if $\Ho_i(N,K)$ is finite dimensional over $K$ for every $i \leq m$. 
 Using homological methods we  show the following surprising result.

\medskip
\noindent {\bf Theorem A.} {\it  Let $L = \oplus_{i \geq 1} L_i$ be an $\mathbb{N}$-graded Lie algebra  that is $FP_{n}$ such that   $[L,L]= \oplus_{i \geq 2} L_i $ and $M$ be a proper ideal of $L$ such that $[L,L]\leq M$.
	 Then  $M$ is of type $FP_n$ if and only if for every Lie subalgebra  $N$
	of $L$ of codimension one such that $M \subseteq N$  we have that $N$ is $FP_n$.}

\medskip
Since an $\mathbb{N}$-graded Lie algebra $L$ is of type $FP_1$ (resp. $FP_2$) if and only if it is finitely generated (resp. finitely presented) Theorem A implies  immediately the following corollary.

\medskip
\noindent {\bf Corollary B.} {\it  Let $L = \oplus_{i \geq 1} L_i$ be an $\mathbb{N}$-graded Lie algebra  that is finitely generated (resp. finitely presented) such that   $[L,L]= \oplus_{i \geq 2} L_i $ and $M$ be a proper ideal of $L$ such that $[L,L]\leq M$.
	Then  $M$ is finitely generated (resp. finitely presented) if and only if for every Lie subalgebra  $N$
	of $L$ of codimension one such that $M \subseteq N$  we have that $N$ is finitely generated (resp. finitely presented).}

\medskip
Theorem A together with a result of  Wasserman \cite{Wasserman}, whose proof uses the HNN-construction for Lie algebras,  imply another surprising result.

\medskip

\noindent {\bf Corollary C.} {\it  Let $L = \oplus_{i \geq 1} L_i$ be an $\mathbb{N}$-graded Lie algebra  that is finitely presented (in terms of generators and relations). Assume that   $[L,L] = \oplus_{i \geq 2} L_i $ and that $L$ does not contain an ordinary non-abelian free Lie subalgebra. Then $[L,L]$ is a finitely generated Lie algebra.}

\medskip 
 
There are examples of finitely presented metabelian Lie algebras 
$L$ with infinitely generated $[L,L]$ if $\mathrm{char}(K) \not= 2$ (for example see  Claim 1 from subsection 5.3 in \cite{K-M1}   but such examples are not $\mathbb{N}$-graded). Thus the hypothesis of Corollary C  that $L$ is $\mathbb{N}$-graded  is   indispensable.

 Bestvina and Brady constructed in \cite{B-B} groups of type $\FP_2$ that are not finitely presented. For a finite  graph $\Gamma=(V(\Gamma),E(\Gamma))$ (graphs are assumed to be without loops or multiple edges all throughout the paper) the right angled Artin group $G_{\Gamma}$ is generated by the vertices $V(\Gamma)$ of $\Gamma$  with relations $[v_1, v_2] = 1$, whenever $v_1$ and $v_2$ are linked by an edge in $E(\Gamma)$. The Bestvina-Brady group $H_{\Gamma}$ is the kernel of the map $G_{\Gamma} \to \mathbb{Z}$ that sends every vertex to 1. Let $\Delta_{\Gamma}$ be the flag complex of $\Gamma$, i.e.,  the complex obtained from $\Gamma$ after gluing a simplex to every clique (complete subgraph) of $V(\Gamma)$. By the main result in \cite{B-B}  $H_{\Gamma}$ is finitely presented if and only if $\Delta_\Gamma$ is 1-connected.  A combinatorial proof of the backward direction was given by  Dicks and Leary in  \cite{D-L}. Furthermore by \cite{B-B} $H_{\Gamma}$ is of type $\FP_m$ if and only if $\Delta_\Gamma$ is $(m-1)$-acyclic. These results were later generalized to kernels of arbitrary homomorphisms (often called real characters) $G\to\R$ by Meier, Meinert and van Wyk (\cite{M-M-W}, see below).

A graph $\Gamma$ as before also determines a Lie ring (over $\mathbb{Z}$) defined as
 $$\GG_{\Gamma} = Fr (V(\Gamma)) / \langle \langle [v,w] = 0 \hbox { if } v, w \hbox{ are end points of an edge from } E \rangle \rangle,$$ where $Fr(V(\Gamma))$ is the free Lie ring (over $\mathbb{Z}$) with a free basis the set $V(\Gamma)$ and $\langle \langle R \rangle \rangle$ denotes the ideal generated by $R$. 
 The lower central series of a group $G$ yields a Lie ring $\gr(G)$  (over $\mathbb{Z}$) $$\gr( G) = \oplus_{k \geq 1} \gamma_k(G)  /
\gamma_{k+1}(G) $$ 
 which is   $\mathbb{N}$-graded and generated by elements in degree one. In the case of a right angled Artin group $G_\Gamma$, there is a natural isomorphism of Lie rings  $\gr( G_{\Gamma}) \simeq \GG_{\Gamma}$, see \cite{D-K}, \cite{D-K2},  \cite[Thm.~3.4]{P-S2}. Moreover,  Papadima and Suciu showed  in
   \cite[Cor.~9.6]{P-S3}  that if    $H_1(\Delta_{\Gamma}, \mathbb{Q}) = 0$, the inclusion map $H_{\Gamma} \to G_{\Gamma}$ induces a group isomorphism $H_{\Gamma}' \to G_{\Gamma}'$ and an isomorphism of the derived subalgebras of the $\mathbb{N}$-graded Lie algebras $\gr(G_{\Gamma})\otimes_\Z \Q$   and $\gr(H_{\Gamma})\otimes_\Z \Q $.  The approach of Papadima and Suciu  uses  the theory of 1-formality   \cite{P-S}, \cite{P-S2}, \cite{P-S3}. We will not use techniques involving 1-formality in this paper but observe that 1-formality requires $\Char(K) = 0$. The results in this paper are proved for  a field $K$ of  an arbitrary characteristic. 
 
Consider the $K$-Lie algebra  $L_{\Gamma} = \gr(G_{\Gamma})\otimes_\Z K=\GG_\Gamma\otimes_\Z K$ for a fixed field $K$.
Let $\pi:L_\Gamma\to L_\Gamma/[L_\Gamma,L_\Gamma]$ be the canonical projection and consider a non-zero linear map  $\chi:L_\Gamma/[L_\Gamma,L_\Gamma]\to K$. Put $I_{\chi} = \pi^{-1}(\ker(\chi))$.
The living graph $\Gamma_{\chi}$ is the full subgraph of $\Gamma$ spanned by the vertices with non-zero $\chi$-value. Let $\Delta_{\Gamma_{\chi}}$ be the flag complex of $\Gamma_{\chi}$.  
For a (possibly empty) clique $w\subseteq\Gamma$, $\lk_{\Delta_\Gamma}(w)$  is defined as  either $\Delta_\Gamma$ if $w=\emptyset$ or the link in  $\Delta_\Gamma$ of the simplex associated to $w$  otherwise. We also set 
 	$$ \lk_{\Delta_{{\Gamma}_{\chi}}}(w) : = \lk_{\Delta_{\Gamma}}(w ) \cap \Delta_{{\Gamma}_{\chi}}.$$

 Note that any codimension 1 ideal of $L_\Gamma$ is of the form $I_\chi$ for some $\chi$. 	In the following result we use homological methods to classify  when those ideals are of type $FP_n$. 
 A similar computation for Bestvina-Brady groups can be found in \cite{L-M} and for cocyclic normal subgroups of certain Artin groups in \cite{BCMP}.	
	 	
 	\medskip
\noindent 	{\bf Theorem D.} {\it  The Lie algebra $N = I_{\chi}$ is of type $FP_n$ if and only if $~ \lk_{\Delta_{\Gamma_{\chi}}}(w)$ is  $(n-1 - |w| )$-acyclic  over the field $K$ for every (possibly empty) clique $w\subseteq\Gamma\setminus\Gamma_\chi$. For $w= \emptyset$ this translates to the flag complex $\Delta_{\Gamma_{\chi}}$ is $(n-1)$-acyclic  over $K$ .}
 	
 	\medskip
	
	 This result together with Theorem A above yield a 
a complete classification of the coabelian ideals $M$ of $L=L_\Gamma$ which are  of type $FP_n$.  A group theoretic version for right angled Artin groups  was proved by Meier, Meinert and van Wyk in \cite{M-M-W} as a corollary of their description  of the  Bieri-Neumann-Strebel-Renz $\Sigma$-invariants for  right angled Artin groups. 
  In Theorem \ref{groupsbigrank} in Subsection \ref{section-groups}  we give a group theoretic version of Theorem A for the right angled Artin group  $G_{\Gamma}$ that involves the Bieri-Neumann-Renz-Strebel $\Sigma$-theory. 

We state two corollaries of  Theorem D . The first of them, Corollary E below, characterizes when the ideal $I_{\chi}$ is finitely generated as a Lie algebra.
Although it can be deduced from Theorem A, in Section \ref{appendix1}  we  show how to prove it using only elementary combinatorial methods. For the statement, recall that we say that a subgraph $\Gamma_1$ of $\Gamma$ is {\it dominant} in $\Gamma$ if for every $v \in V(\Gamma) \setminus V(\Gamma_1)$ there is $w \in V(\Gamma_1)$ that is linked with $v$ by an edge in $\Gamma$. This definition is important for the description of the Bieri-Neumann-Strebel-Renz $\Sigma^1$-invariant of a right angled Artin group $G$ given by Meier and van Wyk in \cite{Meier}, which was later generalised in \cite{M-M-W}.

\medskip

\noindent  {\bf Corollary E.} {\it   
The ideal $I_{\chi}$  of $L_{\Gamma}$  is finitely generated as a Lie algebra if and only if  $\Gamma_{\chi}$ is connected and  dominant in $\Gamma$. In this case $I_{\chi}$ is generated as a Lie algebra by $\Ker(\chi)$.
	}

\medskip 

Our second corollary to  Theorem D    can be viewed, in the specific case $m = 2$, as a Lie algebra version of the results of Dicks and Leary in \cite{D-L}. 

\medskip 

\noindent {\bf Corollary F.} {\it The $\mathbb{N}$-graded Lie algebra $\gr(H_{\Gamma})\otimes_\Z K $ is of type $\FP_m$ 
	if and only if the flag complex $\Delta_{\Gamma}$  is $(m-1)$-acyclic over $K$, i.e.  $\Ho_i(\Delta_{\Gamma},K) = 0$  for $ i\leq m-1$. } 

\medskip

In particular, Corollary F  classifies when $\gr(H_{\Gamma})\otimes_\Z K $ is a finitely presented Lie algebra since for $\mathbb{N}$-graded Lie algebra $FP_2$ and finite presentability are the same.
Note that by \cite{B-B} we know that  the group  $H_{\Gamma}$ is finitely presented if and only if the flag complex  $\Delta_{\Gamma}$ is 1-connected and by Corollary F  $\gr(H_{\Gamma})\otimes_\Z K$  is a finitely presented Lie algebra if and only if the flag complex  $\Delta_{\Gamma}$ is 1-acyclic. In the case $K = \mathbb{Q}$ this was earlier proved using 1-formality by Papadima and Suciu in \cite{P-S3} .

Finally in Section \ref{appendix2}  we use geometric methods to prove the following result,  which gives a sufficient condition for  a coabelian ideal $M$ of a right angled Artin Lie algebra $L_\Gamma$  to be of type $FP_n$. 

\medskip

\noindent {\bf Theorem G. }{\it Let $L = L_{\Gamma}$ and  $[L,L]\leq M\normal L$ be an ideal of codimension $k$. Let $X$ be the set of all subsets of $V(\Gamma)$ which generate complete subgraphs and $Y\subseteq X$ the set of those $A\in X$ such that $M\cap L_A$ has corank $k$ in $L_A$. Assume that
$\lk_{|Y|}(Z)$ is $(n-i-1)$-acyclic for any $Z\in X \setminus Y$ with $|Z|=i$. Then $M$ is of 
type $\FP_n$.}

\medskip
{\bf Acknowledgements} During the preparation of this work the first named author was partially supported by CNPq grant 301779/2017-1  and by FAPESP grant 2018/23690-6. The second named author was partially supported by PGC2018-101179-B-I00 and by Grupo \'Algebra y Geometr\'ia, Gobierno de Arag\'on and Feder 2014-2020 \lq\lq Construyendo Europa desde Arag\'on".

\section{On the homological finiteness property $FP_m$ of $\mathbb{N}$-graded Lie algebras}\label{homgrad}

We denote by $\mathbb{N}$ the set of positive natural numbers and $\mathbb{N}_0 = \mathbb{N} \cup \{ 0 \}$. An $\mathbb{N}$-graded Lie algebra $L$ is the direct sum $\oplus_{i \in \mathbb{N}} L_i$, where each $L_i$ is $K$-vector subspace of $L$ and $[L_i, L_j] \subseteq L_{i+j}$. Furthermore if each $L_i$ is finite dimensional we call $L$ an $\mathbb{N}$-graded Lie algebra of {\it finite} type. Note that every finitely generated $\mathbb{N}$-graded Lie algebra $L$ is of finite type. Furthermore an $\mathbb{N}$-graded Lie algebra $L$ is finitely generated if and only $L/ [L, L] \simeq H_1(L, K)$ is finite dimensional.

Note that when $L$ is an $\mathbb{N}$-graded Lie algebra there is a natural $\mathbb{N}_0$-grading on the universal enveloping algebra $R = U(L) $ as $\oplus_{i \in \mathbb{N}_0} R_i$ where $R_0 = K.1$ and $R_i = \sum_{i_1 + \ldots + i_j = i, j \geq 1   }L_{i_1} \ldots L_{i_j}$. An $\mathbb{N}_0$-graded $R$-module $V = \oplus_{i \in \mathbb{N}_0} V_i$ has the property that $V_i R_j \subseteq V_{i+j}$. From now we write graded $R$-module for $\mathbb{N}_0$-graded $R$-module and a homomorphism
$\varphi : V = \oplus_{i \geq 0} V_i \to W = \oplus_{i \geq 0} W_i$ between graded $R$-modules is called graded if it is a homomorphism of $R$-modules and preserves the grading i.e. $\varphi(V_i) \subseteq W_i$. Here and all along the paper, module means {\it right} module.

One important property of  a graded $R$-module $V$ is that we have the following version of the Nakayama Lemma: For an $R$-module $V$, $V = 0$ if and only if $V \otimes_{R} K = 0$ (see \cite{Weigel}, note that although in \cite{Weigel} all graded Lie algebras and modules are assumed to be of finite type this property remains true in general, i.e., if the graded objects are not of finite type).
 Furthermore, $V$ is finitely generated as $R$-module if and only if $V \otimes_{R} K$ is finite dimensional  (over $K$) , see \cite[Section 2]{Weigel}. 
 
 \begin{lemma} \label{FP2} Let $L$ be an $\mathbb{N}$-graded Lie algebra. Then $L$ is of type $FP_2$ if and only if $L$ is finitely presented (in terms of generators and relations).
 	\end{lemma}

 \begin{proof} The non-trivial direction is to show that when $L$ is $FP_2$ then $L$ is finitely presented.
 	Note that for every Lie algebra $L$ of type $FP_2$ the homology groups $\Ho_1(L, K) \simeq L/ [L,L]$ and $\Ho_2(L, K)$ are finite dimensional.  In particular, since $L/ [L,L]$ is finite dimensional, $L$ is finitely generated and hence  $L$ is an $\mathbb{N}$-graded Lie algebra of finite type.
 	   By \cite[(2.13)]{Weigel} if an $\mathbb{N}$-graded Lie algebra $L$  of finite type  has finite dimensional $\Ho_2(L, K) = \Tor_2^{U(L)}(K,K)$ then $U(L)$ is a finitely presented graded associative algebra, hence $L$ is a finitely presented Lie algebra.
 \end{proof}

Let $L$ be an $\mathbb{N}$-graded Lie algebra. We say that $L$ is {\it graded} $FP_m$ if there is a graded projective resolution ( i.e. of graded $U(L)$-modules with graded homomorphisms) of the trivial $U(L)$-module  $K$, where all projective modules in dimension smaller or equal to $m$ are finitely generated.

Let $L = \oplus_{i \geq 1} L_i$ be an $\mathbb{N}$-graded Lie algebra. An ideal $N$ of $L$ is a {\it graded} ideal if $N = \oplus_{i \geq 1} N \cap L_i$.

The following is a Lie graded algebra version of the pro-$p$-groups result \cite[Thm.~A]{King}. 

\begin{proposition}  \label{homgrad2} Let $L$ be an $\mathbb{N}$-graded Lie algebra over a field $K$ and $N$ be a graded ideal such that $U(L/ N)$ is left and right Noetherian. Then
	$L$ is graded $FP_m$ if and only if $\Ho_i(N, K)$ is finitely generated as $U(L/ N)$-module for every $i \leq m$.
This implies that both graded $FP_m$ and ordinary $FP_m$ are the same property.	\end{proposition} 

\begin{proof} Suppose first that 	$L$ is graded $FP_m$.  Then $L$ is ordinary $FP_m$.  Any projective resolution ${\mathcal P}$ of the trivial $U(L)$-module $K$ with finitely generated modules in dimension up to $m$ can be used to calculate the homology groups $\Ho_i(N,K) = \Ho_i({\mathcal P} \otimes_{U(N)} K)$. Since ${\mathcal P} \otimes_{U(N)} K$ is a complex of projective $U(L/N)$-modules and up to dimension $m$ all modules are finitely generated, using Noetherianess we deduce that 
	$\Ho_i(N,K)$ is finitely generated as $U(L/ N)$-module for every $i \leq m$.
	
	For the converse we built by induction on $m \geq 0$ a free graded resolution ${\mathcal P}$ of the trivial $U(L)$-module $K$ with finitely generated modules in dimension up to $m$. The case $m = 0$ is obvious. Suppose we have built the resolution up to dimension $m-1$, then we have an exact graded  complex of right $U(L)$-modules ( i.e. all modules and homomorphisms are graded) 
	$$
	{\mathcal P} : 0 \to B_m \to P_{m-1} \to \ldots \to P_1 \to P_0 \to K \to 0,
	$$
	 where $P_0, \ldots, P_{m-1}$ are free $U(L)$-modules.
	Consider the complex ${\mathcal P} \otimes_{U(N)} K$.  We claim that
	\begin{equation} \label{30godini}
	\Ker (B_m \otimes_{U(N)} K \to P_{m-1} \otimes_{U(N)} K ) \simeq \Ho_m(N,K).
	\end{equation}
	Indeed let
	$$
	{\mathcal Q} : \ldots \to Q_{m+1} \to Q_{m} \to P_{m-1} \to \ldots \to P_1 \to P_0 \to K \to 0
	$$
	be a free graded resolution of $U(L)$-modules (i.e. all homomorphisms are graded and the modules are free).
	Since tensoring is right exact, the exactness of $Q_{m+1} \to Q_m \to B_m \to 0$ implies that 
	$$Q_{m+1} \otimes_{U(N)} K \Vightarrow{\alpha} Q_m \otimes_{U(N)} K \Vightarrow{\beta} B_m \otimes_{U(N)} K \to 0 \hbox{ is exact},$$
	 hence $$\im (\alpha) = \Ker (\beta).$$ Note that
	$$
	\Ho_m(N,K) = \Ker (Q_m \otimes_{U(N)} K \to P_{m-1} \otimes_{U(N)} K )/ \im ( \alpha ) = $$ $$ \Ker (Q_m \otimes_{U(N)} K \to P_{m-1} \otimes_{U(N)} K )/ \Ker ( \beta ) \simeq$$ $$ \beta ( \Ker (Q_m \otimes_{U(N)} K \to P_{m-1} \otimes_{U(N)} K ) ) = \Ker (B_m \otimes_{U(N)} K \to P_{m-1} \otimes_{U(N)} K ),
	$$ hence (\ref{30godini}) holds.

	 Note that the  $U(L/N)$-module  $P_{m-1} \otimes_{U(N)} K$ is finitely generated  and since $U(L/N)$ is Noetherian $\im (B_m \otimes_{U(N)} K \to P_{m-1} \otimes_{U(N)} K )$ is finitely generated  too. This combined with (\ref{30godini}) implies that
	$B_m \otimes_{U(N)} K$ is a finitely generated $U(L/N)$-module, hence
	 \begin{equation} \label{dimension} B_m \otimes_{U(L)} K   \hbox{  is finite dimensional over } K . \end{equation} 
	  Recall that all modules and homomorphisms in $\mathcal P$ are graded, in particular $B_m$ is a graded $U(L)$-module.  Thus (\ref{dimension}) implies that $B_m$ is a finitely generated graded $U(L)$-module. Then we can choose $P_m$ as any finitely generated graded free $U(L)$-module together with a graded epimorphism $P_m \to B_m$.
	  
	   Finally if $L$ is ordinary $FP_m$ then the first paragraph of the proof implies that $\Ho_i(N, K)$ is finitely generated as $U(L/ N)$-module for every $i \leq m$. Then $L$ is graded $FP_m$.
\end{proof}
Applying the proposition for $N = L$ we obtain

\begin{corollary} \label{fpm} Let $L$ be an $\mathbb{N}$-graded Lie algebra over a field $K$. Then 
		$L$ is of homological type $FP_m$ if and only if $\Ho_i(L, K)$ is finite dimensional for every $i \leq m$.
	\end{corollary}

 \subsection{Cohomological finiteness of coabelian ideals}
 
 \ \ \ \ \ \ \ \ \ \ \ \ \ \ \ \ \ \ \ \ \ \ \ \ \ \ \ \ \ \ \ \ \ \ \ \ \ \ \ \ \ \ \ \ \ 
     
 Let $M$ be a coabelian ideal of an $\N$-graded Lie algebra $L$, i.e., an ideal such that $[L,L]\leq M$. In this subsection we will prove that the cohomological finiteness properties of $M$ can be determined by the cohomological finiteness properties of all those codimensional one ideals $N$ of $L$ such that $M\leq N$.

\begin{proposition} \label{prop-embed} Let $R = K[x_1, \ldots, x_k]$ and $S = K[x_2, \ldots, x_k]$ be commutative polynomial rings.   We view $S$ as a left $R$-module via a surjective homomorphism of $K$-algebras $\theta : R \to S$ that 
 maps $Kx_1 + \ldots + K x_k$ surjectively  to $Kx_2 + \ldots + K x_k$. Let ${\mathcal P}$ be a complex
	$$
	{\mathcal P} : \ldots \Vightarrow{}  P_i ~ \Vightarrow{\partial_i} ~ P_{i-1}  \Vightarrow{} \ldots \Vightarrow{} P_0 \Vightarrow{} 0
	$$
	of free $R$-modules. Then $\Ho_i(\mathcal{P}) \otimes_{R}S$  embeds in 
	$\Ho_i({\mathcal{P} \otimes_{R} S})$.
	\end{proposition}

\begin{proof}

Note that since $P_i$ is a free $R$-module we have an exact sequence 
$$ 
0 = \Tor_1^R( P_i, S) \to \Tor_1^R( \im (\partial_i) , S) \to   \ker(\partial_i) \otimes_R S \Vightarrow{\alpha_i} P_i \otimes_R S  \Vightarrow{\beta_i} \im(\partial_i) \otimes_R S \to 0$$	
  which is part of the long exact sequence in $\Tor^R_*(- , S)$ applied to  the short exact sequence of $R$-modules 
$$0 \to \ker(\partial_i) \to P_i \to \im(\partial_i) \to 0.$$
Using $\theta$ we can construct a free resolution of $S$ as $R$-module 
$$0 \to R \Vightarrow{\mu}  R \Vightarrow{\theta} S \to 0,$$ to see it note that $\ker (\theta) = y R \not= 0$ for some  $y \in K x_1 + \ldots + K x_k$   and set $\mu(r) = y r$.    We  use this resolution to calculate $\Tor_1^R( \im (\partial_i) , S)$ i.e.
$$
\Tor_1^R( \im (\partial_i) , S) = \ker ( \im (\partial_i)  \Vightarrow{\nu_i}  \im (\partial_i)   ),$$
where $\nu_i$ is given by multiplication with $y$. Since $ \im (\partial_i) $ is an $R$-submodule of the free $R$-module  $P_{i-1} $  we deduce that $\ker(\nu_i) = 0$, hence $\Tor_1^R( \im (\partial_i) , S) = 0$ and $\alpha_i$ is injective.

The short exact sequence of $R$-modules  $0 \to \im(\partial_{i+1}) \to \ker(\partial_i) \to \Ho_i(\mathcal{P}) \to 0$ induces an exact sequence
$$
\im(\partial_{i+1}) \otimes_R S \Vightarrow{\gamma_i} \ker(\partial_i) \otimes_R S \to \Ho_i(\mathcal{P}) \otimes_R S \to 0.
$$ 
Then
$$
\Ho_i(\mathcal{P}) \otimes_R S \simeq ( \ker(\partial_i) \otimes_R S ) / \im(\gamma_i).$$
	Let $\{ d_i =  \partial_i \otimes id_S \} $  be the differentials of
$	{\mathcal P} \otimes_S R$.
Then
$$d_{i+1} =  \alpha_i \gamma_i \beta_{i+1}$$
and using that $\alpha_{i-1}$ is injective and $\beta_{i+1}$ is surjective we get
$$\Ho_i( 	{\mathcal P} \otimes_S R ) = \ker(d_i)/ \im (d_{i+1}) = \ker (  \alpha_{i-1} \gamma_{i-1} \beta_{i})/ \im (   \alpha_i \gamma_i \beta_{i+1} )  =$$ $$ 
\ker( \gamma_{i-1} \beta_{i} ) / \im (\alpha_i \gamma_i) \supseteq \ker (\beta_i) / \im (   \alpha_i \gamma_i) = \im (\alpha_i) /  \im (   \alpha_i \gamma_i) \simeq$$ $$ (\ker(\partial_i) \otimes_R S )/ \im (\gamma_i) \simeq \Ho_i(\mathcal{P}) \otimes_R S.
$$
	\end{proof}

 Let $R = K[x_1, \ldots, x_k]$  be a commutative polynomial ring. We can see $R$ as a graded ring, where  $x_1, \ldots, x_k$ have weight 1.

\begin{proposition}  \label{prop-homology}  Let  $i \geq 0$  be a fixed natural number. Let   $R = K[x_1, \ldots, x_k]$ be a commutative polynomial ring and  ${\mathcal P}$ be a graded complex
	$$
	{\mathcal P} : \ldots \Vightarrow{}  P_i ~ \Vightarrow{\partial_i} ~ P_{i-1}  \Vightarrow{} \ldots \Vightarrow{} P_0 \Vightarrow{} 0
	$$
	of free $R$-modules, where each $P_j$ is finitely generated for $j \leq i$. 
	Suppose that for every	$K$-algebra epimorphism $\theta : R \to S_0 = K[v]$, where $\theta(x_s) \in K v$ for $ 1 \leq s \leq k$, we have that 	$\Ho_i({\mathcal{P} \otimes_{R} S_0})$ is finite dimensional over $K$, where
	we view $S_0$ as a left $R$-module via $\theta$. Then
	 the homology group $\Ho_i(\mathcal{P})$
is finite dimensional over $K$. 
	\end{proposition}

\begin{proof} 
	From the very beginning we can assume that $K$ is an algebraically closed field, otherwise we consider the complex ${\mathcal P} \otimes_K \overline{K}$ over $\overline{K}[x_1, \ldots, x_k]$, where $\overline{K}$ is the algebraic closure of $K$.  Then  $\Ho_i({\mathcal P} \otimes_K \overline{K}) \simeq \Ho_i({\mathcal P}) \otimes_K \overline{K}$ is finite dimensional over $\overline{K}$ if and only if 
	$\Ho_i({\mathcal P})$ is finite dimensional over $K$.
	
	In order to prove the proposition we proceed by induction on $k$. The case $k = 1$ is obvious as  $S_0 = R$ and  we can choose $\theta$ to be the identity map.
	Assume from now on that $k \geq 2$ and that the proposition holds for all polynomial rings on at most $k-1$ variables.

	Let 
	$$\begin{aligned}
	R_0 = R \Vightarrow{\theta_0} R_1 = K[x_2, \ldots, x_k] \Vightarrow{\theta_1} \ldots \to R_i = K[x_{i+1}, \ldots, x_k] \Vightarrow{\theta_i}\\
	 \to R_{i+1} = K[ x_{i+2}, \ldots, x_k] \to  \ldots \Vightarrow{\theta_{k-2}} R_{k-1} = K[x_k]\end{aligned}$$ be epimorphisms of $K$-algebras, where for each epimorphism $\theta_j$ we have  $\theta_j(K x_{j+1} + \ldots + K x_k) = K x_{j+2} + \ldots + K x_k $. Then by Proposition \ref{prop-embed}  $\Ho_i(\mathcal{P} \otimes_R R_j) \otimes_{R_j} R_{j+1}$ embeds in 
	$\Ho_i({\mathcal{P} \otimes_{R} R_{j+1}})$. By the assumptions of  Proposition \ref{prop-homology} for $S_0 = R_{k-1}$ and $\theta = \theta_{k-2}  \ldots \theta_0$ we deduce that 	$\Ho_i({\mathcal{P} \otimes_{R} R_{k-1}})$ is finite dimensional over $K$.

	On the other hand, if $j \geq 1$ we consider ${\mathcal{P} \otimes_{R} R_{j}}$ as a complex of free $R_j$-modules and as $R_j$ is a polynomial ring on less than $k$ variables and for such rings by induction Proposition \ref{prop-homology}  holds  we conclude  that 
	 $\Ho_i({\mathcal{P} \otimes_{R} R_{j}})$ is finite dimensional over $K$.
	 \medskip
	We claim that this implies the result. To see it, note that by Proposition \ref{prop-embed}  $\Ho_i(\mathcal{P}) \otimes_{R} R_{1}$ embeds in the finite dimensional $K$-vector space
	 $\Ho_i({\mathcal{P} \otimes_{R} R_{1}})$. Then the short exact sequence 
	 $$0 \to \Ho_i(\mathcal{P}) \otimes_{R} R_{1} \to \Ho_i({\mathcal{P} \otimes_{R} R_{1}}) \to Q \to 0,$$ where $Q = \Ho_i({\mathcal{P} \otimes_{R} R_{1}})/\Ho_i(\mathcal{P}) \otimes_{R} R_{1}$ is finite dimensinal,  induces a long exact sequence
	 $$\begin{aligned}
	 \ldots \to \Tor^{R_1}_1(Q, R_{k-1} ) \to (\Ho_i(\mathcal{P}) \otimes_{R} R_{1})  \otimes_{R_1} R_{k-1} \to\\ \Ho_i({\mathcal{P} \otimes_{R} R_{1}}) \otimes_{R_1} R_{k-1}  \to Q  \otimes_{R_1} R_{k-1} \to 0.
	 \end{aligned}$$
	 Note that since $R_{k-1}$ is a cyclic $R_1$-module, by Noetherianess we can  find a free resolution of $R_{k-1}$ as $R_1$-module, where each module is finitely generated. Using this resolution to calculate  $\Tor^{R_1}_1(Q, R_{k-1} ) $ together with the fact that $Q$ is finite dimensional we deduce that $\Tor^{R_1}_1(Q, R_{k-1} ) $ is finite dimensional. Since $R_{k-1}$ is a cyclic $R_1$-module and $\Ho_i({\mathcal{P} \otimes_{R} R_{1}})$ is finite dimensional over $K$ we deduce that  $\Ho_i({\mathcal{P} \otimes_{R} R_{1}}) \otimes_{R_1} R_{k-1}$ is finite dimensional over $K$. This combined with the above long exact sequence implies that $(\Ho_i(\mathcal{P}) \otimes_{R} R_{1})  \otimes_{R_1} R_{k-1} \simeq \Ho_i(\mathcal{P}) \otimes_{R}  R_{k-1}$ is finite dimensional over $K$ too.
	 
	 Suppose that $\Ho_i(\mathcal{P})$ is not finite dimensional. By Noetherieness   $\Ho_i(\mathcal{P})$ is a finitely generated $R$-module and as the complex $\mathcal{P}$ is graded,  $\Ho_i(\mathcal{P})$ is a graded $R$-module. Hence $\Ho_i(\mathcal{P})$ has an infinite dimensional graded quotient $M = R/ I$, i.e. $I$ is a graded ideal in $R$.  Moreover, again by Noetherianess,  $I$ is finitely generated so it is generated by a finite set of homogeneous polynomials $f_1, \ldots, f_t$. 
	 For each $1\leq s\leq k$ consider the multiplicatively closed set $\Delta_s = \{ x_s^z \mid z\geq 0 \}$. Then 
	 $$
	 M \Delta_s^{-1} = K[y_1, \ldots, y_{s-1}, y_{s+1}, \ldots, y_k,  y_s^{\pm 1}  ] / I \Delta_s^{-1},$$
	 where $y_j = x_j/x_s$ for $j \not= i$ and $y_s = x_s$. Then since each $f_j$ is homogeneous we can write $f_j = y_s^{d_s} g_j(y_1, \ldots, y_{s-1}, y_{s+1}, \ldots, y_k)$ for  some positive integer $d_s$ and some polynomial $g_j$ for every $ 1 \leq j \leq t$. Write $J_s$ for the ideal of $K[y_1, \ldots, y_{s-1}, y_{s+1}, \ldots, y_k]$ generated by $g_j$ for $1 \leq j \leq t$. Then
	 $$  M \Delta_s^{-1} \simeq (K[y_1, \ldots, y_{s-1}, y_{s+1}, \ldots, y_k]/ J_s) \otimes_K K[y_s^{\pm 1}].$$  If the localization $M \Delta_s^{-1}$ is not zero,  
	 we embed $J_s$ in a maximal ideal of ${K}[y_1, \ldots, y_{s-1},$ $ y_{s+1}, \ldots, y_k]$ and since $K$ is algebraically closed any maximal ideal is generated by $\{ y_j - \lambda_j\}_{j\not= s}$ for some $\lambda_j \in K$. Hence there is an epimorphism of ${K}$-algebras
	 $$\rho_s:  M \Delta_s^{-1}  \to {K}[y_s^{\pm 1}]$$ that sends $y_j$ to $\lambda_j$ for $j \not= s$ and $y_s$ to $y_s$. Note that this map sends $x_j$ to $\lambda_j x_s$ for $j \not= s$ and is the identity on $x_s = y_s$. The canonical map $M \to M \Delta_s^{-1}$  composed with $\rho_s$ gives a graded  epimorphism
	 $$\overline{\rho}_s:  M  \to {K}[y_s] = K[x_s].$$
	  Consider the epimorphism $\gamma = \overline{\rho}_s \circ \pi : R \to K[x_s]$, where $\pi : R \to M = R/ I$ is the canonical projection. Thus  we can view $K[x_s]$ as $R$-module via $\gamma$ and there is an epimorphism of $K$-vector spaces $$ \overline{\rho}_s  \otimes id : M \otimes_{R } {K}[x_s] \to {K}[x_s] \otimes_R {K}[x_s].
	 $$ Furthermore consider the epimorphism
	  $$
	  \pi \otimes id : R \otimes_R K[x_s] \to M \otimes_R K[x_s]
	  $$ and note that  after identifying $R \otimes_R K[x_s]$ with $K[x_s]$ and identifying  ${K}[x_s] \otimes_R {K}[x_s]$ with $K[x_s]$ (via the multiplication in $K[x_s]$) , we get that $
	  \pi \otimes id$ is the inverse of $ \overline{\rho}_s  \otimes id$. In particular $ M \otimes_{R } {K}[x_s] \simeq K[x_s]$  and $M \otimes_{R } {K}[x_s]$ 
	  is infinite dimensional as $K$-vector space.

	  Recall that $\Ho_i(\mathcal{P}) \otimes_R R_{k-1} = \Ho_i(\mathcal{P}) \otimes_R K[x_k] $ is finite dimensional. By permuting the variables $x_1, \ldots, x_k$, we have that $ \Ho_i(\mathcal{P}) \otimes_R K[x_s] $ is finite dimensional, where $K[x_s]$ is a left $R$-module via an arbitrary epimorphism $R \to K[x_s]$ that sends $K x_1 + \ldots + K x_k$ onto $K x_s$. Note that $M \otimes_{R } {K}[x_s] \simeq  {K}[x_s]$  is a quotient of the finite dimensional  $\Ho_i(\mathcal{P}) \otimes_R K[x_s]$,  a contradiction.

	 Hence for every $1 \leq s \leq k$ we have that $M \Delta_s^{-1} = 0$. Thus there exists a positive integer $z_s$ such that $x_s^{z_s} \in I$, so $x_s \in \sqrt{I}$. Then
	 $$(x_1, \ldots, x_k) \subseteq \sqrt{I},$$
	 hence the radical $\sqrt{I}$ has finite codimension in $R = K[x_1, \ldots, x_k]$ (in fact codimension 1). Note that by Noetherianess there is a positive integer $z$ such that $\sqrt{I}^{~ z} \subseteq {I}$ and each $\sqrt{I}^j/ \sqrt{I}^{j+1}$ is a finitely generated $R/ \sqrt{I}$-module, hence is finite dimensional. Thus each $R/ \sqrt{I}^j$ is finite dimensional and so $M = R/ I$ is finite dimensional, a contradiction.	 
	\end{proof}

We can now prove the main result of this section, which is Theorem A from the introduction.

\begin{theorem} \label{geral2} {\bf (Theorem A)} Let $L$ be an $\mathbb{N}$-graded Lie algebra of type $FP_{n}$ such that  $L = \oplus_{i \geq 1} L_i$ and $[L,L] = \oplus_{i \geq 2} L_i $ and $M$ be a proper ideal of $L$ such that $[L,L]\leq M$. Then  $M$ is of type $FP_n$ if and only if for every Lie subalgebra  $N$
	of $L$ such that $M \subseteq N$ and $\dim_{K} (L/N) = 1$ we have $N$ is $FP_n$.
\end{theorem}

\begin{proof} The easy part of the proof is to show that when $M$ is $FP_n$ then each $N$ is $FP_n$. Indeed since $N/ M$ is a finite dimensional abelian Lie algebra, it is $FP_{\infty}$, in particular is $FP_n$.  This follows from the fact that $U(N/M)$ is a Noetherian ring.  Then $N$ is an extension of $M$, an $FP_n$ Lie algebra,  by $N/M$. As $N/M$ is also of type $FP_n$, $N$ is $FP_n$ as claimed.
	
	Suppose now that each $N$ is of type $FP_n$.  Note that since $M$ is an ideal of $L$ such that $\oplus_{i\geq 2}L_i=[L,L]\leq M$ then $M$ is also $\mathbb{N}$-graded via
	$M=(L_1\cap M) \oplus (\oplus_{i\geq 2}L_i).$
	By Corollary \ref{fpm}  $M$ is of type $FP_n$ if and only if $\Ho_i(M,K)$ is finite dimensional for $i \leq n$. 
	
	Let $ {\mathcal{R}}$ be a free graded resolution of the trivial $U(L)$-module $K$ with finitely generated modules in dimension $\leq n$. Set
	$$\mathcal{P} =  {\mathcal{R}}^{del} \otimes_{U(M)} K,$$
where the upper index del means that we have substituted the module in dimension -1 with the zero module. This $\mathcal{P}$ is a complex of free $R$-modules, where $R =  U(L/M)$ is a  commutative  polynomial ring.
	Then for  $i \geq 0$ 
	$$\Ho_i(M, K) \simeq \Ho_i(\mathcal{P})$$  and $$\Ho_i(N,K) \simeq \Ho_i({\mathcal{R}}^{del} \otimes_{U(N)} K) \simeq \Ho_i(\mathcal{P} \otimes_{U(N/M)} K) \simeq \Ho_i(\mathcal{P} \otimes_{U(L/M)} U(L/N) ).$$
	Since $N$ is $FP_n$ we have that 
	$\Ho_i(\mathcal{P} \otimes_{U(L/M)} U(L/N) )$ is finite dimensional for every $i \leq n$ and every $N$ of codimension 1 in $L$. Then by Proposition \ref{prop-homology} this implies that $\Ho_i(\mathcal{P})$ is finite dimensional for $i \leq n$.	
	\end{proof}

\begin{corollary} {\bf (Corollary C)} Let $L$ be a finitely presented $\mathbb{N}$-graded Lie algebra $L = \oplus_{i \geq 1} L_i$  that does not contain an ordinary non-abelian free Lie subalgebra and $[L,L] = \oplus_{i \geq 2} L_i$. Then $[L,L]$ is a finitely generated Lie algebra.
	\end{corollary}

\begin{proof} 
	
	By \cite[Cor.~9.2]{Wasserman} if $L$ is a finitely presented soluble Lie algebra then every ideal $N$ of codimension 1 is finitely generated as a Lie subalgebra. The proof of \cite[Cor.~9.2]{Wasserman} needs solubility only to exclude the possibility that $L$ contains an ordinary non-abelian free Lie subalgebra.  Thus we can use the same argument here and we can apply Theorem \ref{geral2} for $M = [L,L]$ and $n = 1$ together with the fact that  a Lie algebra is finitely generated if and only if it is of type $FP_1$. 
	\end{proof}

\section{Lie subalgebras of type $FP_n$ in $L_{\Gamma}$}

Let $\Gamma$ be a finite graph with no loops or double edges. Denote by  $x_1, \ldots, x_m$ the set of vertices $V(\Gamma)$ of $\Gamma$ and let $G=G_{\Gamma}$ be the right angled Artin group associated with  $\Gamma$ and $$L = L_{\Gamma} = \gr(G_{\Gamma})\otimes_\Z K.$$
 Let $$V = K x_1 + \ldots + K x_m$$ be a $K$-vector space with a basis $x_1, \ldots, x_m$. We identify $V $ with $ L/ [L,L]$  and let $$\chi : V \to K$$  be  a non-zero $K$-linear map. Put $$I_{\chi} = \pi^{-1} (\Ker (\chi)),$$ where $\pi :  L  \to  L / [L,L]$ is the canonical epimorphism. For example if $\chi (x_i) = 1$ for all $i$ and $\Gamma$ is connected then $I_{\chi}$ is precisely  $\gr(H_{\Gamma})\otimes_\Z K$, where $H_{\Gamma}$ is the kernel the epimorphism $G\to \mathbb{Z}$ that sends each $x_i$ to 1.  This follows from the fact that for a connected finite graph $\Gamma$  by \cite[Thm.~5.6]{P-S} the inclusion map $H_{\Gamma} \to G_{\Gamma}$ induces an isomorphism of Lie algebras $[gr(H_{\Gamma}), gr(H_{\Gamma})] \simeq [gr(G_{\Gamma}), gr(G_{\Gamma})]$.   And in the general case $I_\chi$ is a codimension one ideal of $L$.  Moreover, since	 $L = \oplus_{i \geq 1} L_i$ and $[L,L] = \oplus_{i \geq 2} L_i$ we deduce that $I_\chi= (I_\chi \cap L_1) \oplus (\oplus_{i \geq 2} L_i)$
 is an $\mathbb{N}$-graded Lie algebra.

 We fix an order in the set the vertices of $V(\Gamma)$.  
 Recall that the flag complex $\Delta_{\Gamma}$  associated to the graph $\Gamma$ is the complex obtained from $\Gamma$ after gluing a simplex to every non empty clique of $\Gamma$, i.e. an $n$-cell of $\Delta_{\Gamma}$ is $(v_1, \ldots, v_n)$, where the vertices $v_1, \ldots, v_n$ are all pairwise linked in $\Gamma$ and $v_1<v_2<\ldots<v_n$.
 
 The following complex is the well known minimal resolution of the trivial $U(L)$-module $K$  (see \cite{K-M1} Subsection 2.4).

 \begin{equation}\label{complexArtin}
{\mathcal P}_{\Gamma} :  \ldots \Vightarrow{\partial_{n+1}} P_n \Vightarrow{\partial_n} P_{n-1} \Vightarrow{\partial_{n-1}} \ldots \Vightarrow{\partial_1} P_0 \Vightarrow{\partial_{0}}  K \to 0,\end{equation}
where 
$$P_n = \bigoplus_{\sigma\subseteq\Gamma } c_{\sigma} U(L)$$
and the  sum is over all the cliques  $\sigma$ of $\Gamma$ with $|\sigma|=n$. Here each $c_{\sigma} U(L)$ is a copy of the free right $U(L)$-module, $c_{\emptyset} = 1_K$, $P_0 = c_{\emptyset} U(L) = U(L)$ and $\partial_0$ is the augmentation map. For higher degrees  the differential is
$$
\partial_n(c_\sigma) =  \sum_r (-1)^{r-1} c_{\sigma \setminus \{ v_{i_r} \} } v_{i_r},
$$
  where $\sigma = ( v_{i_1}, \ldots, v_{i_n})$.

 The link  of a vertex $v$ in a graph is the subgraph spanned by all those vertices different from $v$ which are linked to it. This definition extends to subsets of vertices by taking the intersection of all the links of the vertices in the subset. In the case of the empty set, the link is the graph itself.
This definition can be extended to simplicial complexes: the link of a simplex $s$ in a simplicial complex $C$ is the subcomplex of $C$ consisting of the simplices $t$ that are disjoint from $s$ and such that both $s$ and $t$ are faces of some higher-dimensional simplex in $C$, equivalently, such that $s\cup t$ is also a simplex in $C$. 

  Consider the living graph $\Gamma_{\chi}$ i.e. the subgraph of $\Gamma$ spanned by the vertices with non-zero $\chi$-value. Let $\Delta_{\Gamma_{\chi}}$ be the flag complex of $\Gamma_{\chi}$ and $\Delta_{\Gamma}$ be the flag complex of $\Gamma$. The flag complex of the link of a non-empty clique $w$ in $\Delta_{\Gamma}$ is the simplicial link of the simplex spanned by the clique. We denote this complex by $\lk_{\Delta_\Gamma}(w)$ (also in the case when $w=\emptyset$).  Fix a possibly empty clique  $w\subseteq \Gamma\setminus\Gamma_\chi$ and set
$$ \lk_{\Delta_{{\Gamma}_{\chi}}}(w) = \lk_{\Delta_{\Gamma}}(w ) \cap \Delta_{{\Gamma}_{\chi}}.$$

 At this point, we can prove Theorem C.  Recall that a space $W$ is $m$-acyclic over a field $K$ if $\Ho_i(W, K) = 0$ for all $ 0 \leq i \leq m$. 

\begin{theorem} \label{geral} {\bf (Theorem D)} The Lie algebra $N = I_{\chi}$ is of type $FP_n$ if and only if  $\lk_{\Delta_{\Gamma_{\chi}}}(w )$ is  $(n-1 - |w|)$-acyclic   over $K$ for every clique $w\subseteq\Gamma\setminus\Gamma_\chi$. For $w = \emptyset$ this translates to the flag complex $\Delta_{\Gamma_{\chi}}$ is $(n-1)$-acyclic  over $K$ .
\end{theorem}

\begin{proof}  For $i \geq 1$  the homology $\Ho_i(N,K)$ is the homology of the complex ${\mathcal P}_{\Gamma}\otimes_{U(N)}K$, where ${\mathcal P}_{\Gamma}$ is the complex in (\ref{complexArtin}). To describe this complex, note first that 
$$P_n\otimes_{U(N)}K = \bigoplus_{\sigma\subseteq\Gamma\text{ clique, }|\sigma|=n} c_{\sigma} U(L)\otimes_{U(N)}K \simeq C_n\otimes_KU(L/N),$$
where 
$$C_n=\bigoplus_{\sigma\subseteq\Gamma\text{ clique, }|\sigma|=n} Kc_{\sigma}.$$
As for the differential $d_n:=\partial_n\otimes 1_d$, identifying $U(L)\otimes_{U(N)}K \simeq U(L/N)$ with the polynomial ring $K[v]$, where for each $v_i\in V(\Gamma)$, $v_i\otimes 1 \in U(L)\otimes_{U(N)}K$  is sent to $\chi(v_i) v$ we have 
$$
d_n(c_\sigma\otimes 1) =  \sum_r (-1)^{r-1} c_{\sigma \setminus \{ v_{i_r} \} } v_{i_r}\otimes 1=\sum_r (-1)^{r-1} c_{\sigma \setminus \{ v_{i_r} \} } \otimes\chi(v_{i_r})v,
$$
where $\sigma=(v_{i_1},\ldots,v_{i_n})$.

To get an easier description of $d_n$, we renormalize the $K[v]$-modules $C_nK[v]$ by setting for each clique $\sigma\subseteq\Gamma$
$$c_\sigma = \widetilde{c}_\sigma \prod_{v\in\sigma,\chi(v)\neq 0} \chi(v).
$$
From now on we write $C_n K[v]$ instead of $C_n \otimes_{\mathbb{Z} } K[v]$, and think of $C_n K[v]$ as a free $K[v]$-module with free basis $\{ c_w \}$, i.e. the free basis of $C_n$.
Then 
$$ 
d_n(\widetilde{c}_\sigma)=\sum_{  \chi(v_{i_r})\neq 0   } (-1)^{r-1}  \widetilde{c}_{\sigma \setminus \{ v_{i_r} \} } v.
$$ 
 Define  $D_n=\im(d_{n+1})$ and $A_n=\Ker(d_n)$. Then $$\Ho_n(N,K)=A_n/D_n.$$
 
 Consider the chain complex $$\mathcal{C} :  \ldots \to C_i \to C_{i-1} \to \ldots \to C_0 = K \to 0,$$
with differential $$
\bar{d}_n:C_n\to C_{n-1}$$ given by 
$$\bar{d}_n( \widetilde{c}_w) = \sum_{  \chi(v_{i_r})\neq 0   } (-1)^{r-1} \widetilde{c}_{\sigma \setminus \{ v_{i_r} \} },$$ thus 
$\bar{d}_n( \widetilde{c}_w)v=d_n( \widetilde{c}_w)$.

 Denote by $\mathcal{C}_\bullet$ the complex obtained from $\mathcal{C}$ by shifting the index by -1 i.e. in $\mathcal{C}_\bullet$ the module $C_n$ is in dimension $n - 1$ and we write $\widehat{d}_n$ for its differential, thus $\widehat{d}_{n-1} = \bar{d}_n$. Note that    $$D_n=\im(\bar{d}_{ n+1   })vK[v] \hbox{ and }A_n=\Ker(\bar{d}_n)K[v].$$ 
Define $B_n=\im(\bar{d}_{ n+1  } )K[v]$. Then we have a short exact sequence of $K$-vector spaces
 	\begin{equation} \label{ses--1} 0 \to B_n/ D_n \to A_n / D_n \to A_n / B_n \to 0.\end{equation}
 Note that
	$B_n/ D_n \simeq  \im (\overline{d}_{  n+1 }) = \im (\widehat{d}_{n})$  is a $K$-vector subspace of the finite dimensional $K$-vector space $C_{ n  }$, hence  $ \im (\widehat{d}_{n})$  is finite dimensional. Moreover
	$$A_n / B_n \simeq  (\ker (\widehat{d}_{n-1}) / \im (\widehat{d}_{n}))   \otimes_{K } K[v] \simeq \Ho_{n-1}(\mathcal{C}_\bullet) \otimes_{K} K[v].$$ 
	Hence the short exact sequence (\ref{ses--1}) gives
	 a short exact sequence
	$$0 \to  \im (\widehat{d}_{n})  \to \Ho_n (N, K) \to \Ho_{n-1}(\mathcal{C}_\bullet) \otimes_{K} K[v]\to 0$$
	so 
	$\Ho_n (N, K) $ is finite dimensional if and only if $  \Ho_{n-1}(\mathcal{C}_\bullet)= 0$.
	
As $N$ is $\mathbb{N}$-graded, Corollary \ref{fpm} implies that $N$ is of type $FP_n$ if and only if $\Ho_i(N,K)$ is finite dimensional for $i \leq n$.
By the previous paragraph this is equivalent to the $(n-1)$-acyclicity of  the complex  $\mathcal{C}_\bullet$.

Fix a clique $w=(v_{k+1}, \cdots, v_n) \subseteq\Gamma\setminus\Gamma_\chi$, i.e., all its  vertices have zero $\chi$-value and
consider the subcomplex $\mathcal{C}^w_\bullet$ of $\mathcal{C}_\bullet$ spanned by those $ \widetilde{c}_\sigma $, where 
$$\sigma = (v_1, \ldots,v_k,v_{k+1},\ldots, v_n )$$
is a clique in $\Gamma$
with $v_1, \ldots, v_k$ satisfying  $\chi(v_i) \not= 0$ and  $w=(v_{k+1}, \cdots, v_n)$ the fixed clique.
  Then $\mathcal{C}_\bullet$ is a direct sum of the subcomplexes  $\mathcal{C}^w_\bullet$ over all possible
cliques $w\subseteq\Gamma\setminus\Gamma_\chi$.

Note that   $\mathcal{C}^\emptyset_\bullet$ is the chain complex of the flag complex $\Delta_{\Gamma_{\chi}}$ of the living graph $\Gamma_{\chi}$ and that $\widetilde{c}_\sigma  \in \mathcal{C}^w_\bullet$
if and only if  $\sigma_0 =(v_1, \ldots, v_k)$ is a simplex in $\lk_{\Delta_{{\Gamma}_{\chi}}}(w)$. Hence 
 $$\Ho_{  n-1 }(\mathcal{C}^w_\bullet) \simeq \Ho_{ n-1 - | w |      }( \lk_{\Delta_{{\Gamma}_{\chi}}}(w), K)$$
and
$$\Ho_{  n-1 }(\mathcal{C}_\bullet) \simeq \oplus_{w ~ } 
\Ho_{n-1 - | w |      }( \lk_{\Delta_{{\Gamma}_{\chi}}}(w), K),$$ 
where the sum is over cliques $w\subseteq\Gamma\setminus\Gamma_\chi$.

 Therefore  $\mathcal{C}_\bullet$ is $(n-1)$-acyclic if and only if 
$\Ho_i( \lk_{\Gamma_{\chi}}(w) ) \otimes_{\mathbb{Z}} K = 0$ for  $i \leq n-1  - |w| $  for  any clique $w\subseteq\Gamma\setminus\Gamma_\chi$   i.e. if and only if $\lk_{\Gamma_{\chi}}(w)$ is   $(n-1 - | w | )$-acyclic  over $K$  for any clique $w\subseteq\Gamma\setminus\Gamma_\chi$  .

\end{proof}

\begin{remark} The condition on the links of the statement of Theorem \ref{geral} (= Theorem D) is equivalent with the acyclicity condition used in the statement of \cite[Main ~Thm.]{M-M-W} that classifies when $\chi : G_{\Gamma} \to \mathbb{Z}$ belongs to the Bieri-Neumann-Renz-Strebel invariant $\Sigma^n(G_{\Gamma}, \mathbb{Z})$. Indeed, there is an obvious modification  of the proof of Theorem \ref{geral} for groups, where $N$ is a subgroup of $G_{\Gamma}$ with $G_{\Gamma} / N \simeq \mathbb{Z}$. For a field $K$ this implies that $\Ho_i(N, K)$ is finite dimensional for $i \leq n$ precisely when the acyclicity condition used in the statement of Theorem \ref{geral} holds. By \cite[Thm.~7.3]{P-S3} $\Ho_i(N, K)$ is finite dimensional for $i \leq n$ if and only if $N$ is of type $FP_n$ (note this is true in this specific case for a subgroup $N$ satisfying $G_{\Gamma} / N \simeq \mathbb{Z}$ and is not a general statement). But by \cite[Cor.~A]{M-M-W}  $N$ is of type $FP_n$ if and only if the  acyclicity condition used in the statement of \cite[Main ~Thm.]{M-M-W} holds.
\end{remark}

\begin{corollary} {\bf (Corollary F)} Let $N$ be the Lie algebra  $\gr(H_{\Gamma})\otimes_\Z K$ and $\Delta_{\Gamma}$ be the flag complex of $\Gamma$. Then  $N$ is of type $FP_n$ if and only if  $\Delta_{\Gamma}$ is $(n-1)$-acyclic over $K$ i.e. $\Ho_i(\Delta_{\Gamma}, K) = 0$ for $i \leq n-1$.	
\end{corollary}

\begin{proof} 
	Note that $N = I_{\chi}$, where $\chi(v) = 1$ for every $v \in V(\Gamma)$. Then $\Gamma = \Gamma_{\chi}$ and by Theorem \ref{geral} $N$ is $FP_n$ if and only if the flag complex $\Delta_{\Gamma} = \lk_{\Delta_{\Gamma_{\chi}}}(\emptyset)$ is $(n-1)$-acyclic over the field $K$.
\end{proof}

We recall next the statement of Theorem \ref{geral2} for right angled Artin Lie algebras $L_\Gamma$. This result together with Theorem \ref{geral} gives a complete classification of the coabelian ideals of $L_\Gamma$ which are of type $\FP_n$.

\begin{corollary} \label{neu} Let  $L= L_{\Gamma} = \gr(G_{\Gamma})\otimes_\Z K$ and $M$ be a proper ideal of $L$ such that $[L,L]\leq M$. Then  $M$ is of type $FP_n$ if and only if for every Lie subalgebra   $I_{\chi}$
	of $L$ such that $\chi(M) = 0$  we have that $\lk_{\Gamma_{\chi}}(w)$ is   $(n-1- |w| )$-acyclic  over $K$  for any clique $w\subseteq\Gamma\setminus\Gamma_\chi$.
\end{corollary}

\subsection{Kernels of higher codimension : the group case} \label{section-groups}

\ \ \ \ \ \ \ \ \ \ \ \ \ \ \ \ \ \ \ \ \ \ \ \ \ \ \ \ \ \ \ \ \ \ \ \ \ \ \ \ \ \ \ \ \ \ 

Next we prove the group theoretical version of Corollary \ref{neu}  which involves the Bieri-Neumann-Strebel-Renz $\Sigma$-invariants of a group $G$. By definition the $\Sigma$-invariants are subsets  of the character sphere $S(G) = Hom(G, \mathbb{R})/ \sim$, where two characters (non-zero homomorphisms) $\chi_1, \chi_2 : G \to \mathbb{R}$ are equivalent  $\chi_1 \sim \chi_2$ if and only if there is a positive real number $r$ such that $\chi_1 = r \chi_2$. The equivalence class of $\chi$ is denoted by $[\chi]$. If $G/ [G,G]$ has torsion-free rank $n$ it is easy to see that $S(G)$ can be identified with the unit sphere $S^{n-1}$ in $\mathbb{R}^n$.
We omit details about the $\Sigma$-invariants but note that there are known for few classes of groups including the class of right angled Artin groups, see \cite[Main~Thm.]{M-M-W}. Another fact that we will need in the proof of the following theorem is that the equivalence of items i) and iii) is precisly the statement of \cite[Thm.~B]{B-R}. In this subsection we use the $\Sigma$-invariants as a technical tool we need in order to prove that items i), ii) and iv) from Theorem \ref{groupsbigrank}  are equivalent.

Recall that a group $G$ is of type $FP_n$ if the trivial $\mathbb{Z} G$-module $\mathbb{Z}$ has a projective resolution with finitely generated modules in dimensions smaller or equal to $n$. If such projective resolution exists then there is a free resolution with the same property. 
 In the case when $G_{\Gamma}/ M \simeq \mathbb{Z}$ the equivalence between i) and ii) in Theorem \ref{groupsbigrank} was proved in \cite[Thm.~7.3]{P-S3}.
 
 \begin{theorem}\label{groupsbigrank} For a normal subgroup $M$ of the right angled Artin group $G_{\Gamma}$ that contains the commutator the following conditions are equivalent:
 \begin{itemize}
 \item[i)]  $M$ is of homological type $FP_n$,
 
 \item[ii)]  $\Ho_i(M, K)$ is finite dimensional (over  $K$) for $i \leq n$ and every field $K$,
 
 \item[iii)] $[\chi]\in\Sigma^n(G_\Gamma,\mathbb{Z})$ for any character $\chi:G_\Gamma\to\R$ that vanishes on $M$,
 
  \item[iv)] for every subgroup $N$ of $G$ such that $M \subseteq N$ and $G_{\Gamma} / N \simeq \mathbb{Z}$ we have that $N$ is $FP_n$.
 \end{itemize}
 	\end{theorem}
 
 \begin{proof} 
 	Note that i) is equivalent to iii) by \cite[Thm.~B]{B-R}. 
 	
 	The fact that i) implies ii) is obvious. Indeed, if ${\mathcal P}$ is a free resolution of the trivial $\mathbb{Z} M$-module $\mathbb{Z}$ with finitely generated modules in dimensions smaller or equal to $n$ then $H_i(M, K) \simeq H_i (\mathcal{P} \otimes_{\mathbb{Z} M} K)$ is finite dimensional ( over $K$) since the modules of  $\mathcal{P} \otimes_{\mathbb{Z} M} K$ in dimensions smaller or equal to $n$ are all finite dimensional (over $K$).
 	
 	We check that ii) implies iv).  For any discrete character $\chi : G_{\Gamma} \to \mathbb{Z}$  such that $\chi(M) = 0$, consider the Lyndon-Hochshild-Serre spectral sequence  $$E^2_{p,q} = \Ho_p(\ker(\chi)/ M, \Ho_q(M, K))$$ that converges to $\Ho_{p+q}(Ker(\chi), K)$. Since $\ker(\chi)/ M$ is a finitely generated abelian group and $\Ho_q(M,K)$ is finite dimensional for $q \leq n$,  we deduce that $\Ho_i(\Ker (\chi), K)$ is finite dimensional for all $i \leq n$. By Theorem 7.3 from \cite{P-S3}, $N = \ker(\chi)$ is $FP_n$, i.e. iv) holds. 
 	
 	Finally, we claim that iv) implies iii).  If $\chi  : G_{\Gamma} \to \mathbb{R}$ is a  character (i.e. non-zero homomorphism)  such that $\chi(M) = 0$, let $\chi_0 :  G_{\Gamma} \to \mathbb{Z}$  be a discrete character such that $\chi_0(M) = 0$ and  $\Gamma_{\chi} = \Gamma_{\chi_0}$ i.e. for $v \in V(\Gamma)$ we have $\chi(v) = 0 $ if and only if $ \chi_0(v) = 0$. By \cite[Thm.~B]{B-R} the fact that $N = \ker(\chi_0)$ is $FP_n$ implies $[\chi_0] \in \Sigma^n(G_\Gamma,\mathbb{Z})$. Then by the description of $\Sigma^n(G_\Gamma, \mathbb{Z})$ in \cite[Main~Thm.]{M-M-W} the property $\Gamma_{\chi} = \Gamma_{\chi_0}$  implies $[\chi] \in \Sigma^n(G_\Gamma,\mathbb{Z})$ if and only if $[\chi_0] \in \Sigma^n(G_\Gamma,\mathbb{Z})$. Thus iii) holds.
 	\end{proof}

\section{Elementary approach for finite generation of $I_{\chi}$} \label{appendix1}  

In this Section  we give a  proof of the criterion when $I_{\chi}$ is finitely generated as a Lie algebra, i.e. Corollary E, using more elementary combinatorial methods and avoiding homological arguments, i.e. not using Theorem D. 
 
 \medskip
 
 As before $L_{\Gamma} = \gr(G_{\Gamma})\otimes_\Z K$ for a graph $\Gamma$.
 
 \begin{lemma} \label{points} Let $\Gamma$ be a graph with $m\geq 2$ vertices and no edges and  $$\chi :V\to K$$ be a non-zero $K$-linear map for $V=L_\Gamma/[L_\Gamma,L_\Gamma]$. Then $I_{\chi}$ is not finitely generated as Lie algebra.
 \end{lemma}
 
 \begin{proof} Note that in this case $G_{\Gamma}$ is the free group of rank $m$ and $\gr(G_{\Gamma})\otimes_\Z K$ is the free Lie algebra  $F_n$ (over $K$) on $m$ elements. And no proper ideal in $F_n$ is finitely generated as a Lie algebra \cite[Theorem~3]{B}. 
 \end{proof}
 
 Recall that $\Gamma_{\chi}$ is the subgraph of $\Gamma$ spanned by all vertices with non-zero $\chi$-value. We call $\Gamma_\chi$ the living subgraph with respect to $\chi$. We say that a subgraph $\Gamma_1$ of $\Gamma$ is {\bf dominant} in $\Gamma$ if for every $v \in V(\Gamma) \setminus V(\Gamma_1)$ there is $w \in V(\Gamma_1)$ that is linked with $v$ by an edge in $\Gamma$.
 
 \begin{lemma} 
 	\label{corE-1} Let $\Gamma$ be a finite graph.

 	a) Suppose that $\Gamma_{\chi}$ is not connected. Then $I_{\chi}$ is not finitely generated as a Lie algebra.
 	
 	b)  Suppose that $\Gamma_{\chi}$ is connected but is not dominant in $\Gamma$. Then $I_{\chi}$ is not finitely generated as a Lie algebra.
 \end{lemma} 
 
 \begin{proof} a) Let  $\Gamma_0$ be the set of connected components  in $\Gamma_\chi$. We view $\Gamma_0$ as a graph without edges and denote by $[v]$ the connected component in $\Gamma_0$ of a vertex $v\in\Gamma_\chi$.  For every $[v]\in\Gamma_{\chi}$ choose one vertex $v_0 \in [v]$. Define
 	an epimorphism of Lie algebras $$\pi_0 : L_{\Gamma} \to L_{\Gamma_0} $$ 
 	given by $\pi_0(w) = \frac{\chi(w)}{\chi(v_0)} [v]$ if $w \in [v] \in \Gamma_0$  and  $\pi_0(w) = 0$ if $\chi(w) = 0$ .
 	Then $\ker (\pi_0) \subseteq I_{\chi}$.  
 	
 	Note that $\Gamma_0$ is a finite set of  at least two disjoint points, hence we can apply Lemma \ref{points} to deduce that $I_{\chi_0}$ is not finitely generated, where $$\chi_0:  \oplus_{w \in \Gamma_0} K w \to K$$ is  a $K$-linear map  induced by $\chi$, i.e. $\chi_0([v]) = \chi(v_0).$ Note that $I_{\chi} / \ker (\pi_0) \simeq I_{\chi_0}$, hence $I_{\chi}$ is not finitely generated.
 	
 	b) Let $v_1$ be a vertex of $\Gamma \setminus \Gamma_{\chi}$ that is not linked with { any} vertex from $\Gamma_{\chi}$. Consider the epimorphism $$\pi_1: L_{\Gamma} \to L_{\Gamma_1},$$ where $\Gamma_1$ is the subgraph spanned by $\Gamma_{\chi}$ and $v_1$ and $\pi_1$ sends every vertex of $\Gamma \setminus (\Gamma_{\chi} \cup \{ v_1 \})$ to 0. Then $\Ker (\pi_1) \subseteq I_{\chi}$ and $I_{\chi}/ \ker (\pi_1) \simeq I_{\chi_1}$, where $$\chi_1 : L_{\Gamma_1} / [L_{\Gamma_1},L_{\Gamma_1}] \to K$$  is the restriction of $\chi$.  
 	
 	Recall that  $\Gamma_{\chi}$ is a connected graph. Let $\Gamma_2$ be the graph with two vertices $u_1$ and $u_2$ and no edges. 
Fix a vertex $v_0\in\Gamma_\chi$ and define the epimorphism of Lie algebras 
$$\begin{aligned}\pi_2 : L_{\Gamma_1} &\to L_{\Gamma_2}\\  
v_1&\mapsto u_1, \ \ 
v&\mapsto{\chi(v)\over\chi(v_0)}  u_2 	\text{ for any }v\in\Gamma_\chi.\\
\end{aligned}$$ 
 
 	Note that $\ker(\pi_2) \subseteq I_{\chi_1}$. Finally we define  a $K$-linear map  $$\chi_2 : L_{\Gamma_2} / [L_{\Gamma_2}, L_{\Gamma_2}]  \to K$$ as induced by $\chi$ i.e.  $\chi_2(u_1) = 0$, $\chi_2(u_2) = \chi(v_0)$. Then $I_{\chi_1} / \ker (\pi_2) \simeq I_{\chi_2}$ and 
 	by 	  Lemma \ref{points} $I_{\chi_2}$ is not finitely generated. Hence $I_{\chi_1}$ and $I_{\chi}$ are not finitely generated.
 \end{proof}
 
 \begin{proposition} \label{corE-2} If $\Gamma_{\chi}$ is connected and dominant in $\Gamma$ then $I_{\chi}$ is generated as a Lie algebra by a basis of the  $K$-linear space  $\ker (\chi)$, in particular $I_{\chi}$ is a finitely generated Lie algebra.
 \end{proposition}
 
 \begin{proof}  Let $T$ be a maximal tree in $\Gamma_{\chi}$. Fix a vertex $x_1 $ in $T$ and decompose $T$ as a union of geodesics $\gamma_i$ that start at $x_1$. Let $v_{i,1} = x_1, v_{i,2}, \ldots, v_{i,s}$ be the consecutive vertices of the geodesics $\gamma_i$. Then define
 	\begin{equation} \label{equ001}
 	y_{i,j}  = \chi(v_{i,j-1}) v_{i,j} - \chi(v_{i,j}) v_{i,j-1} \in \Ker (\chi) \subseteq I_{\chi}
 	\end{equation}  and note that the coefficients $\chi(v_{i,j-1})$ and $\chi(v_{i,j})$ are non-zero since $v_{i,j-1}, v_{i,j}$ are vertices in $\Gamma_{\chi}$. Then by (\ref{equ001}) we can write $v_{i,j}$ as a linear combination $\frac{1}{ \chi(v_{i,j-1})} y_{i,j} + \frac{\chi(v_{i,j})}{ \chi(v_{i,j-1})} v_{i, j-1}$ and then repeat this for $v_{i,j-1}$, $v_{i, j-2}$ and so on until $v_{i,2}$. This gives
 	\begin{equation} \label{eq001}
 	v_{i,j} = \lambda_{i,j} x_1 + \sum_{2 \leq t\leq j} \lambda_{i,j,t} y_{i,t},
 	\end{equation}
 	where all coefficients $\lambda_{i,j}, \lambda_{i,j,t} \in K \setminus \{ 0 \}$. Note that by (\ref{equ001})  $[v_{i,j-1}, v_{i,j}] = 0$  implies that $[v_{i,j-1}, y_{i,j}] = 0$ and using (\ref{eq001}) for  $v_{i,j-1}$ we get
 	$$
 	[ \lambda_{i,j-1} x_1 + \sum_{2 \leq t\leq j-1} \lambda_{i,j-1,t} y_{i,t} ,  y_{i,j}] = 0.$$
 	Thus 
 	\begin{equation} \label{eq002}
 	[x_1, y_{i,j}] \in \langle \{ y_{i,j} \}_{j} \rangle. \end{equation}
 	On other hand for  $x_s \in V(\Gamma) \setminus V(\Gamma_{\chi})$ we have that $\chi(x_s) = 0$ and there is a vertex $v_{i,j} \in V(\Gamma_{\chi})$ such that there is an edge between $x_s$ and $v_{i,j}$. Then $[ v_{i,j}, x_s] = 0$, hence by (\ref{eq001})
 	$$
 	[ \lambda_{i,j} x_1 + \sum_{2 \leq t\leq j} \lambda_{i,j,t} y_{i,t}, x_s] = 0.$$ Then
 	\begin{equation} \label{eq003}
 	[x_1, x_s] \in \langle \{ y_{i,t}, x_s \}_{i,t,s} \rangle.
 	\end{equation}
 	Note that $  \{ y_{i,t} \}_{i,t} \cup \{ x_s \}_{x_s \in V(\Gamma) \setminus V(\Gamma_{\chi})}$ is a basis of  $\Ker (\chi)$ as a $K$-vector space and by (\ref{eq002}) and (\ref{eq003}) the Lie algebra $\langle \Ker(\chi) \rangle $ is an ideal in $L_{\Gamma}$ i.e  it is the ideal $I_{\chi}$.
 \end{proof}

Finally note that  Lemma \ref{corE-1} and Proposition \ref{corE-2}  imply  Corollary E.

\section{Kernels of higher degrees : a sufficient topological condition} \label{appendix2}

 In this Section we explain a more geometric approach that gives a sufficient condition for the Lie algebra $M$ from Theorem \ref{geral2} to be of type $FP_n$ in the case $L = L_{\Gamma}$. 
 Recall that as the universal enveloping algebra $U(L)$ is a Hopf algebra,  the $K$-tensor product of two right  $U(L)$-modules   is a right  $U(L)$-module   via the comultiplication $\Delta : U(L) \to U(L) \otimes U(L)$ that sends $a \in L$ to $1 \otimes a + a \otimes 1$.  Moreover, if $T\leq L$ is a Lie subalgebra, $W$ a $U(T)$-module and $V$ a $U(L)$-module we have the following Mackey-type formula  i.e. an isomorphism of right $U(L)$-modules
 
 \begin{equation}\label{preMackey}
 (W \otimes_K V) \otimes_{U(T)} U(L) \simeq (W \otimes_{U(T)} U(L)) \otimes_K V
 \end{equation}
 that sends $(w \otimes v) \otimes \lambda$ to $\sum_i ( w \otimes \lambda_{i,1}) \otimes v \lambda_{i,2}$, where $\Delta(\lambda) = \sum_i \lambda_{i,1} \otimes \lambda_{i,2}$, (see \cite{K-M1}, pages 8 and 9).

We use the isomorphism (\ref{preMackey}) to show

\begin{lemma}\label{Mackey} Let $M\normal L$ be an ideal of the Lie algebra $L$ and $T\leq L$ a subalgebra. Assume that  the inclusion of $T$ in $L$ induces an isomorphism of Lie algebras $T/(T\cap M) \cong L/ M$. Then there is an isomorphism of $U(L)$-modules
$$(K \otimes_{U(T)} U(L))\otimes_K(K \otimes_{U(M)} U(L)) \simeq K \otimes_{U(T\cap M)} U(L).$$                                                
\end{lemma}
\begin{proof}Note there is an isomorphism of right $U(T)$-modules
$$K \otimes_{U(M\cap T)} U(T)\cong U(T/T\cap M)\cong U(L/M)\cong K \otimes_{U(M)} U(L).$$ Then using (\ref{preMackey})  for $W = K, V = K \otimes_{U(M)}  U(L)$ we obtain
$$\begin{aligned}
(K\otimes_{U(M)} U(L)) \otimes_K(K\otimes_{U(T)} U(L)) \cong ( K \otimes_{U(M)} U(L))\otimes_{U(T)} U(L)
\cong\\
( K \otimes_{U(M\cap T)} U(T))\otimes_{U(T)} U(L)\cong
K \otimes_{U(T\cap M)} U(L).
\end{aligned}$$
\end{proof}

\begin{lemma}\label{resFPn} Let $\Lambda$ be any  associative  ring and $W$ a $\Lambda$-module fitting in an exact chain complex
$$\ker(\delta)\to D_{n}\buildrel{\delta}\over\to D_{n-1}\to\ldots\to D_0\to W \to 0$$
such that each $D_i$ is a $\Lambda$-module of type $\FP_n$. Then $W$ is of type $\FP_n$.
\end{lemma}
\begin{proof} By induction  on  $n$  we may assume that $\ker(D_0\to W)$ is of type $\FP_{n-1}$. Then  there is a short exact sequence of $\Lambda$-modules $0  \to \ker(D_0\to W) \to D_0 \to W \to 0$ with $D_0$ of type $FP_n$ and  $\ker(D_0\to W)$ of type $\FP_{n-1}$ and  
\cite[Prop.~ 1.4]{Bieribook} implies that $W$ is of type $\FP_n$.
\end{proof}

{\begin{lemma}\label{indFPn} (\cite[Lemma 2.8]{K-M1}) Let $S\leq L$ be a subalgebra of an arbitrary Lie algebra $L$. Then $S$ is of type $\FP_n$ if and only if the induced $U(L)$-module $K \otimes_{U(S)} U(L)$ is of type $\FP_n$. 
\end{lemma}}

Recall that in this paper $\Gamma$ always denotes a finite graph and $L_{\Gamma}$ is 
$\gr(G_{\Gamma})\otimes_\Z K$.

\begin{theorem}\label{condition1} Let $ [L,L]  \leq M\normal L$ be a codimension $k$ ideal of the right angled Artin Lie algebra $L=L_\Gamma$. Assume that there is an exact chain complex of $U(L)$-modules
$$\ker(\delta)\to C_n\buildrel{\delta}\over\to C_{n-1}\to\ldots\to C_0\to K  \to 0$$
such that  $K$ is the trivial $U(L)$-module and  each $C_i$ is a finite sum of $U(L)$-modules of the form  
$$K \otimes_{U(T)} U(L)$$
for $T\leq L$  Lie  subalgebras such that $T\cap M$ has codimension $k$ in $T$ and $T\cap M$ is of type $\FP_n$. Then $M$ is also of type $\FP_n$.
\end{theorem}
\begin{proof} Tensoring the exact chain complex in the statement with  $K\otimes_{U(M)}U(L)$  we get an exact chain complex
$$\ker(\tilde{\delta})\to D_n\buildrel{\tilde{\delta}}\over\to D_{n-1}\to\ldots\to D_0\to K \otimes_{U(M)} U(L)  \to 0,$$
where each $D_i$ is a finite sum of modules of the form $(K \otimes_{U(T)} U(L))\otimes_K(K \otimes_{U(M)} U(L))$.
 As $T/T\cap M$ is  a vector space of dimension $k$, we see that $T/T\cap M\cong L/M$ thus using  Lemma \ref{Mackey}  we deduce that each $D_i$ is a finite sum of modules of the form $K \otimes_{U(T\cap M)} U(L)$. By the hypothesis that $T\cap M$ is $\FP_n$ together with  Lemma \ref{indFPn}   these modules are all of type $\FP_n$  as $U(L)$-modules . Then Lemma \ref{resFPn} implies that $K \otimes_{U(M)} U(L)$ is of type $\FP_n$  as $U(L)$-module  and it suffices to apply   Lemma \ref{indFPn}   again  to deduce that $K$ is of type $FP_n$  as $U(M)$-module i.e. $M$ is $FP_n$ . 
\end{proof}

To construct a suitable chain complex to which we can apply the previous Theorem we will use \lq\lq parabolic"  Lie   subalgebras.
 Let $\Omega\subseteq V(\Gamma)$ be a (possibly empty) subset. The {\it parabolic algebra} of $L_\Gamma$ associated to $\Omega$ is the subalgebra of $L_\Gamma$ generated by $\Omega$. It is isomorphic to the right angled Artin Lie algebra $L_{\Gamma_0}$ associated to the full subgraph  $\Gamma_0\subseteq\Gamma$ generated by the vertices in $\Omega$.
Indeed, by \cite[Thm.~6.3]{Wade}  there is an isomorphism $\gr(G_{\Gamma}) \simeq \mathcal{L}(U_{\Gamma})$ of $\mathbb{N}$-graded Lie algebras over $\mathbb{Z}$, where $U_{\Gamma}$ is an associative ring defined as the  free $\mathbb{Z}$-module with basis the  free partially commutative monoid $M_{\Gamma}$ generated by $V(\Gamma)$ and $\mathcal{L}(U_{\Gamma})$ is the Lie ring (over $\mathbb{Z}$) obtained from $U_{\Gamma}$ by defining the Lie operation as $[u_1, u_2] = u_1 u_2 - u_2 u_1$. Note that by \cite{Artin} the embedding of $\Gamma_0$ in $\Gamma$ induces an embedding of $G_{\Gamma_0}$ in $G_{\Gamma}$ and this induces an embedding of $M_{\Gamma_0}$ in $M_{\Gamma}$.
From now on, we will use the same notation, $L_\Omega$ or $L_{\Gamma_0}$ for subsets of vertices $\Omega$ or for subgraphs $\Gamma_0$.

Now, consider the poset of all subsets of $V(\Gamma)$ ordered by inclusion and assume that $H$ is a subposet. Let $|H|$ be the simplicial realization of $H$, that is, the simplicial complex with $k$-simplices $\sigma:A_0\subset A_1\subset\ldots\subset A_k$,  where $A_0, \ldots, A_k \in H$.  We write $| \sigma | = k$,  call $A_0, \ldots , A_k$ vertices of $\sigma$ and $A_k$ the biggest vertex of $\sigma$.  Then  we associate to the simplex $\sigma$ the subalgebra $L_\sigma:=L_{A_0}$, i.e., the subalgebra generated by the smallest of these subsets.
 Let $m$ be the dimension of $|H|$. We define next a chain complex of $U(L)$-modules depending on $H$ as follows. We set
$$C^L(H) :  0 \to  C^L_m(H)\to\ldots\to C^L_k(H)\to\ldots\to C^L_0(H)\to C_{-1}(H)=K \to 0,$$
where $$C^L_k(H)=\oplus_{\sigma\in|H|,|\sigma|= k}  c_\sigma  K \otimes_{U(L_{\sigma})} U(L)$$
and the differential is given by
$$\delta(c_\sigma)=\sum_{i=0}^k(-1)^ic_{\sigma_{i}}$$
where $\sigma_i=\sigma\setminus\{A_i\}$ for $k>0$ and by $\delta(c_{\sigma})=1$ for $k=0$.
We call  $C^L(H)$  the {\sl coset poset}  complex   of $H$ in $L$.

\begin{lemma}\label{full} Let $\Gamma$ be a complete graph, $L=L_\Gamma$ and $X$ the poset of all the subsets of $Z:=V(\Gamma)$. Then $C^L(X)$ is exact and there is a short exact sequence of chain complexes
 $$0\to C^L(X  \setminus Z)\to C^L(X)\to E(Z)\to 0,$$
where $E(Z)$ is a complex with zero homology everywhere except of the top dimension $|Z|$. 
\end{lemma}
\begin{proof} 
1) We show first that $C^L(X)$ is exact. We will define a homotopy 
	$$
	s_k : C^L_k(X) \to C^L_{k+1}(X),
	$$
	which  will be a homomorphism of right $U(L)$-modules.
	For $ \sigma : A_0 \subset A_1 \subset \ldots \subset  A_k$ we write $\sigma = (A_0, \ldots, A_k)$ and we define
	$$
	s_k( c_{ (A_0, \ldots, A_k)  })
	= (-1)^{k+1} c_{(A_0, \ldots, A_k, Z)  } \hbox{ if } A_k \not= Z
	\hbox{ and } 
	s_k( c_{ (A_0, \ldots, A_{k-1}, Z)  }) = 0.
	$$In dimension -1 we define $s_{-1}(1) = c_{Z}$.
	Then for $k \geq 0$ and $A_k \not= Z$ we have
	$$
	(\partial_{k+1} s_k + s_{k-1} \partial_{k}) ( c_{ (A_0, \ldots, A_k)  }) =$$ $$ (-1)^{k+1} \partial_{k+1} ( c_{(A_0, \ldots, A_k, Z)  }) + \sum_{0 \leq i \leq k} (-1)^i s_{k-1} ( c_{(A_0, \ldots, \widehat{A}_i, \ldots, A_k)} )=
	$$	
	$$ \sum_{0 \leq i \leq k} (-1)^{k+1+i}  c_{(A_0, \ldots,\widehat{A}_i, \ldots,  A_k, Z)  } + c_{(A_0, \ldots, A_k)} +  \sum_{0 \leq i \leq k} (-1)^{i+k}  c_{(A_0, \ldots, \widehat{A}_i, \ldots, A_k, Z)}= c_{(A_0, \ldots, A_k)}.
	$$
	And for $k \geq 0$ and $A_k = Z$ we have
	$$
	(\partial_{k+1} s_k + s_{k-1} \partial_{k}) ( c_{ (A_0, \ldots, A_{k-1}, Z)  }) =(s_{k-1} \partial_{k}) ( c_{ (A_0, \ldots, A_{k-1}, Z)  }) =$$
	$$\sum_{0 \leq i \leq k-1} (-1)^i s_{k-1} ( c_{ (A_0, \ldots,\widehat{A}_i, \ldots,  A_{k-1}, Z)  } ) + (-1)^k s_{k-1}( c_{ (A_0,  \ldots,  A_{k-1})  } ) =  c_{ (A_0, \ldots, A_{k-1}, Z)  }.$$
	Finally
	$\partial_{0} s_{-1}(1) = \partial_0(c_Z) = 1 $, hence
	$$\partial_{k+1} s_k + s_{k-1} \partial_{k} = id.$$
	This completes the proof of the fact that  $C^L(X)$ is exact.

2) We will show now that $E(Z)$ has zero homology except at the top dimension $|Z|$. Suppose $Z = V(\Gamma) = \{ x_1, \ldots, x_m \}$ and let $X_{i_1, \ldots, i_s}$ be the poset of all subsets of $Z \setminus \{ x_{i_1}, \ldots, x_{i_s} \}$. Denote by $L_{ i_1, \ldots, i_s }$ the Lie subalgebra of $L$  generated by $Z \setminus \{ x_{i_1}, \ldots, x_{i_s} \}$. Denote by ${\mathcal R}^{del}$ the complex obtained from a complex $\mathcal{R}$ by substituting the module in dimension -1 with 0. Since
$C^{L_{i_1, \ldots, i_s }} (X_{ i_1, \ldots, i_s})$ is an exact complex, the complex $${\mathcal R}_{ \{i_1, \ldots, i_s \}} =   C^{L_{i_1, \ldots, i_s }} (X_{ i_1, \ldots, i_s}) \otimes_{U(L_{i_1, \ldots, i_s })} U(L)   \hbox{ is exact},$$ ${\mathcal R}_{ \{i_1, \ldots, i_s \}}^{del}$ is a subcomplex of $C^L(X \setminus Z)^{del}$ with 
$$ {\mathcal R}_{ \{i_1, \ldots, i_s \}}^{ del  } \cap {\mathcal R}_{ \{j_1, \ldots, i_t \}}^{ del  } = {\mathcal R}^{ del  }_{\{i_1, \ldots, i_s \} \cup  \{ j_1, \ldots, j_t \} } .$$
This induces an exact sequence of complexes
\begin{equation}  \label{exact-novo}
0 \to {\mathcal P} \to {\mathcal R}_{	 \{ x_1, \ldots, x_m \} } \to \ldots \to \oplus_{ I \subseteq Z, | I | = i} {\mathcal{R}}_I \to  \oplus_{ J \subseteq Z, | J | = i-1} {\mathcal{R}}_J \to$$ $$ \ldots \to \oplus_{1 \leq i \leq m} {\mathcal R}_{\{x_i\}} \to C^L(X \setminus Z) \to 0,
\end{equation}
 where for $I = \{ x_{j_1}, \ldots, x_{j_i} \}$ with $j_1 < \ldots < j_i$, $ c_{\sigma} \in {\mathcal{R}}_I $ is send to $\sum_{1 \leq t \leq i} (-1)^t c_{\sigma} \in \oplus_{1 \leq t \leq i} {\mathcal R}_{I \setminus \{ x_t \}}$   and $\mathcal P$ is a complex concentrated in dimension -1.

By dimension shifting argument since (\ref{exact-novo}) is an exact sequence of complexes and all complexes in (\ref{exact-novo})   except the first and the last are exact, we deduce that
$$\Ho_j(\mathcal{P})  \simeq \Ho_{j+m} (C^L(X \setminus Z)).$$
Finally using that $C^L(X)$ is an exact complex and dimension shifting argument for the short exact sequence
 $$0\to C^L(X  \setminus Z)\to C^L(X)\to E(Z)\to 0,$$
 we obtain that
 $$\Ho_{i-1}( C^L(X  \setminus Z) ) \simeq \Ho_{i} (  E(Z) ).$$ Thus
 $$ \Ho_{i} (  E(Z) ) \simeq \Ho_{i-1-m}(\mathcal{P})$$
and since the complex ${\mathcal P}$ is concentrated in dimension -1, $\Ho_* (  E(Z) )$ is concentrated in dimension $m$, as required.
\end{proof}

\begin{proposition}\label{exact} Let $X$ be the poset of all the cliques of $\Gamma$. Then the chain complex $C(X):=C^L(X)$ is exact.
\end{proposition}
\begin{proof} The case when the graph $\Gamma$ is complete is Lemma \ref{full}. In the case when $\Gamma$ is not complete, choose $v_1,v_2\in\Gamma$ not linked  by an edge.  Let $\Gamma_1$ be the subgraph of $\Gamma$  spaned by $V(\Gamma) \setminus \{ v_2\}$ and $\Gamma_2$ be the subgraph of $\Gamma$  spaned by $V(\Gamma) \setminus \{ v_1\}$.  Then $L_\Gamma=L_{\Gamma_1}*_{L_\Lambda}L_{\Gamma_2}$  is  the  amalgamated product of Lie algebras, where $\Lambda = \Gamma_1 \cap \Gamma_2$.  Consider the short exact sequence associated to the amalgamated product. 
\begin{equation}\label{amalg}0\to K \otimes_{U(\Lambda)}  U(L) \to K \otimes_{U(\Gamma_1)}  U(L)\oplus K \otimes_{U(\Gamma_2)}  U(L) \to K  \to 0.\end{equation}
Note that arguing by induction we may assume that if $X_1$ and $X_2$ are the posets of all the subsets of $\Gamma_1$, resp. $\Gamma_2$, that generate complete subgraphs, then the complexes  $C(X_1):=C^{L_{\Gamma_1}}(X_1)$ and $C(X_2):=C^{L_{\Gamma_2}}(X_2)$ are exact. Moreover, $X_1\cap X_2$ is the poset of all the subsets of $\Lambda$ that generate a complete subgraph, so we may also assume that $C(X_1\cap X_2):=C^{L_{\Lambda}}(X_1\cap X_2)$ is exact. From these complexes we can get the exact complexes $C(X_1) \otimes_{U(L_{\Gamma_1})}  U(L) $, $C(X_2) \otimes_{U(L_{\Gamma_2})}  U(L) $ and $C(X_1\cap X_2) \otimes_{U(L_{\Lambda})} U(L)$.
At this point we claim that using (\ref{amalg}) one can show that there is a short exact sequence of chain complexes
$$\begin{aligned}
0\to C(X_1\cap X_2) \otimes_{U(L_{\Lambda})}  U(L) \to( C(X_1) \otimes_{U(L_{\Gamma_1})}  U(L)) \oplus (C(X_2)\otimes_{U(L_{\Gamma_2})}  U(L))\\ \to C(X)\to 0\end{aligned}$$
which implies that $C(X)$ is exact. 
To see this note that at each  dimension $k$   we have a short exact sequence
$$\begin{aligned}
0\to  C_k(X_1\cap X_2) \otimes_{U(L_{\Lambda})}  U(L)  \ 
 \Vightarrow{\iota_k} (C_k(X_1) \otimes_{U(L_{\Gamma_1})}  U(L)) \oplus (C_k(X_2) \otimes_{U(L_{\Gamma_2})}  U(L))\\ \Vightarrow{\pi_k}   C_k(X)\to 0.\end{aligned}$$
Here, $\iota_k$ maps the copy of $(K \otimes_{U(L_\sigma)} U(L_{\Lambda}))  \otimes_{U(L_{\Lambda})} U(L)= K \otimes_{U(L_{\sigma})} U(L)$ corresponding to each $\sigma:A_o\subset\ldots\subset A_k$ in $|X_1\cap X_2|$  to the corresponding sum of copies of $( K \otimes_{U(L_\sigma)} U(L_{\Gamma_i})) \otimes_{U(L_{\Gamma_i})} U(L) =K \otimes_{U(L_{\sigma})} U(L)$, $i=1,2$ via
$a\mapsto (a,-a)$ and $\pi_k$ maps each of the copies in either $C_k(X_1) \otimes_{U(L_{\Gamma_1})} U(L) $ or $C_k(X_2) \otimes_{U(L_{\Gamma_2})} U(L) $ identically to the corresponding copy in $C_k(X)$. 
The facts that $|X|=|X_1|\cup|X_2|$ and that  $|X_1\cap X_2|=|X_1|\cap|X_2|$ imply that this is a short exact sequence. At degree -1 we have the short exact sequence (\ref{amalg}).
Then one easily checks that these maps commute with the differentials of the chain complexes, so the claim follows.
\end{proof}

\begin{proposition}\label{tower} Let $X$ be the poset of all subsets of $V(\Gamma)$ that generate complete subgraphs and $Y\subseteq X$ a subposet with the property that for any $A\in Y$, $B\in X$ with $A\leq B$, we have $B\in Y$. Then there is a finite tower of chain complexes 
$$C^L(Y):=C(Y)=D^0\subseteq D^1\subseteq\ldots\subseteq D^r=C^L(X):=C(X)$$
such that for each $i$ there is a short exact sequence
$$0\to D^{i-1}\to D^i\to T^i\to 0,$$
where $$T^i=\bigoplus_{Z\in X  \setminus Y,|Z|=i}T(Z),$$ 
$$T(Z) \simeq (E(Z) \otimes_{U(L_Z)} U(L))\otimes_K\tilde{C}_{ K }(\lk_{|Y|}(Z)),$$
$\tilde{C}_{ K }$ is the augmented (ordinary) chain complex  over the field $K$  of a simplicial complex  and $E(Z)$ fits in a short exact sequence
$$0\to C^{L_Z}(X_Z \setminus Z)\to C^{L_Z}(X_Z)\to E(Z)\to 0,$$
where $X_Z$ is the poset of all subsets of $Z$  and by definition $E(\emptyset) = K$.  
Moreover,  if $\lk_{|Y|}(Z)$ is $(n-i-1)$-acyclic  over $K$  for any $Z\in X    \setminus  Y$ with $|Z|=i$, then $T(Z)$ is $(n-1)$-acyclic. 
\end{proposition}
  \begin{proof}  Let $D^i$ be the complex associated to the simplices 
$\sigma:A_0\subset\ldots\subset A_k$ such that for $0\leq j\leq k$,  $|A_j|\geq i$ implies $A_j\in Y$. Then 
$$D^0\subseteq D^1\subseteq\ldots\subseteq D^{i}\subseteq D^{i+1}\subseteq\ldots$$ 
Each simplex $\sigma:A_0\subset\ldots\subset A_k$  we define $\sigma_{X \setminus  Y}: A_{0} \subset \ldots \subset A_{s-1}$ and $\sigma_{Y}: A_{s} \subset \ldots \subset A_{k}$, where each $A_{0}, \ldots, A_{s-1}$ is in $X \setminus  Y$ and each $A_s, \ldots, A_k$ is in $Y$.    Then
$D^i$ is the complex of those simplices such that the biggest  vertex $A_{s-1}$ of the simplex   $\sigma_{X \setminus Y}$  is a set of  cardinality at most $i-1$. This implies that if $\sigma \not\in D^i$ but $\sigma  \in D^{i+1}$ then the biggest  vertex   of $\sigma_{X \setminus  Y}$ has cardinality exactly $i$. And from this we see that we may decompose the quotient chain complex $D^{i+1}/D^i$ as a direct sum of 
$$D^{i+1}/D^i=\oplus_{Z\in X \setminus Y,|Z|=i}T(Z),$$
where $T(Z)$ is a sum of the summands corresponding to cells $\sigma$ having precisely $Z$ as the biggest  vertex   in $\sigma_{X \setminus  Y}$. Note that the boundary of such a cell  $\sigma$  will be either a cell in the same set or a cell that lies in $D^i$ and thus vanish in the quotient  $D^{i+1}/D^i$. 

Now, the chain complex $T(Z)$ is the following
$$T(Z)=(E(Z) \otimes_{U(L_Z)} U(L))\otimes_K\tilde{C}(\lk_{|Y|}(Z)).$$
 Recall that $\lk_{|Y|}(Z)$ is a combinatorial subcomplex of $|Y|$ that contais a cell $\sigma_{Y} : A_s \subseteq \ldots \subseteq A_k$ from $|Y|$ if  $\sigma_0 : Z \subset A_s \subset \ldots \subset A_k$ is a cell in $|X|$.

Note that Lemma \ref{full} implies that $\Ho_j(E(Z))=0$ for $j\neq i=|Z|$. 
At this point, using K\"unneth Theorem and the fact that $U(L)$ is flat as $U(L_Z)$-module we have
$$\Ho_n(T(Z))=  \bigoplus_{0\leq j\leq  n  } \Ho_j(E(Z) \otimes_{U(L_Z)} U(L))\otimes_K\tilde{\Ho}_{n-j}(\lk_{|Y|}(Z), K) =  $$ $$\bigoplus_{0\leq j\leq  n  } \Ho_j(E(Z)) \otimes_{U(L_Z)} U(L)\otimes_K\tilde{\Ho}_{n-j}(\lk_{|Y|}(Z),  K ).$$
Thus
$$ \Ho_n(T(Z))= 0 \hbox{ for } n \leq i-1$$ and
 $$ \Ho_n(T(Z))= \Ho_i(E(Z)) \otimes_{U(L_Z)}  U(L)\otimes_K \tilde{\Ho}_{n-i}(\lk_{|Y|}(Z),  K ) \hbox{ for } n \geq i,$$
so if $\lk_{|Y|}(Z)$ is $(n-i-1)$-acyclic  over $K$,  then $T(Z)$ is $(n-1)$-acyclic.

  \end{proof}

\begin{corollary}\label{condition3}  {\bf (Theorem G)} Let $L = L_{\Gamma}$ and  $[L,L]\leq M\normal L$ be an ideal of codimension $k$. Let $X$ be the set of all subsets of $V(\Gamma)$ which generate complete subgraphs and $Y\subseteq X$  be  the set of those $A\in X$ such that $M\cap L_A$ has corank $k$ in $L_A$. Assume that
$\lk_{|Y|}(Z)$ is $(n-i-1)$-acyclic for any $Z\in X \setminus Y$ with $|Z|=i$. Then $M$ is of 
type $\FP_n$.
\end{corollary}
\begin{proof} By Proposition \ref{tower} the complex $C^L(Y) = C(Y)$ is $(n-1)$-acyclic. Then by Theorem \ref{condition1} $M$ is of 
	type $\FP_n$. \end{proof}

Now we want to relate the condition of Corollary \ref{condition3} with the condition of Theorem \ref{geral} in the case when $M$ has codimension 1 in $L = L_{\Gamma} = \gr(G_{\Gamma})\otimes_\Z K$. To do that, let  $\pi : L \to L/ [L,L]$ be the canonical homomorphism and $\chi:L/[L,L]\to K$ be a non-zero  $K$-linear map such that $\pi^{-1}(\ker(\chi)) = M$   and  as in Theorem \ref{geral} denote by $\Gamma_\chi$ the living subgraph of $\Gamma$. Let $X$ and $Y$ be as in Corollary \ref{condition3}.

Let $ Z\in X\setminus  Y$ and for consistency with the notation of Theorem \ref{geral} we set $n-k$ for the cardinality of $Z$. Put $$P_1=\{A\in X\mid A\cap   V (\Gamma_\chi) \neq\emptyset,Z\subseteq A\} $$
 and $$P_2=\{\emptyset\neq B\in X\mid B\subseteq   V(\Gamma_\chi),Z\cup B\in X\}.$$ Let $f:P_1\to P_2$ be given by $f(A)=A\cap V(\Gamma_\chi)$ and $g:P_2\to P_1$ be given by $g(B)=B\cup Z$. Then $fg(B)=B$ and $gf(A)=(A\cap V(\Gamma_\chi))\cup Z\subseteq A$.
Using \cite{Benson} Lemma 6.4.5 this implies that the simplicial realizations $|P_1|$ and $|P_2|$ are homotopy equivalent. 
Note that  $|P_1|=\lk_{|Y|}(Z)$ and $|P_2| = \lk_{\Delta_{\Gamma_{\chi}}}(Z)  = \lk_{\Delta_{\Gamma}}(Z) \cap \Delta_{\Gamma_\chi}$.
This implies that the hypothesis of Corollary \ref{condition3} are equivalent to the hypothesis of Theorem \ref{geral}.

\end{document}